\documentclass[11pt]{amsart}
\usepackage[left=1in, right=1in, top=1.5in, bottom=1.5in]{geometry}
\usepackage{amsfonts,amsmath,amsthm,amssymb}
\usepackage{tikz}
\usetikzlibrary{cd}
\usepackage[mathscr]{euscript}
\usepackage{xcolor}
\usepackage{graphicx, subcaption}	
\usepackage{mathabx}
\usepackage{hyperref}

\graphicspath{{./pics/}}	

\newcommand{\C}{{\mathbb C}}

\newcommand{\D}{{\mathbb D}}
\newcommand{\Z}{{\mathbb Z}}
\newcommand{\R}{{\mathbb R}}
\newcommand{\N}{{\mathbb N}}
\newcommand{\Q}{{\mathbb Q}}

\renewcommand{\Re}{\operatorname{Re}}
\renewcommand{\Im}{\operatorname{Im}}

\newcommand{\id}{\operatorname{id}}

\newcommand{\Post}{\operatorname{Post}}
\newcommand{\PMCG}{\operatorname{PMCG}}
\renewcommand{\id}{\operatorname{id}}

\renewcommand{\phi}{\varphi}

\newcommand{\Teich}{\operatorname{Teich}}

\newcommand{\hide}[1]{}
\newcommand{\Mod}{\operatorname{Mod}}
\newcounter{main}

\theoremstyle{plain}
        \newtheorem{theorem}{Theorem}[section]
        \newtheorem*{theorem*}{Theorem}
        \newtheorem*{conj*}{Conjecture}
        \newtheorem{lemma}[theorem]{Lemma}
        \newtheorem{corollary}[theorem]{Corollary}
        \newtheorem{proposition}[theorem]{Proposition}
        \newtheorem{maintheorem}[main]{Theorem}   
        \newtheorem*{maintheorem*}{Theorem~\ref{thm:A}}

\theoremstyle{definition}
        \newtheorem{definition}[theorem]{Definition}
        \newtheorem*{definition*}{Definition}

\theoremstyle{remark}
        \newtheorem*{remark}{Remark}

        \newtheorem{example}[theorem]{Example}

        \newtheorem*{example*}{Example}
        \newtheorem*{examples*}{Examples}

\theoremstyle{definition}

\title[Dynamical approximation of postsingularly finite exponentials]{Dynamical approximation of postsingularly finite exponentials}

\author[Malavika Mukundan]{Malavika Mukundan}
\address {Dept. of Mathematics, 530 Church Street, University of Michigan Ann Arbor, MI, 48109}
\email{malavim@umich.edu}

\begin{document}
\begin{abstract}
    Given any postsingularly finite exponential function $p_\lambda(z) = \lambda \exp(z)$ where $\lambda \in \C^*$, we construct a sequence of postcritically finite unicritical polynomials $p_{d,\lambda_d}(z) = \lambda_d(1+\frac{z}{d})^d$ that converge to $p_\lambda$ locally uniformly in $\C$, with the same postsingular portrait as that of $p_\lambda$. We describe $\lambda_d$ in terms of parameter rays in the space of degree $d$ unicritical polynomials, and exhibit a relationship between the angles of these parameter rays as $d \rightarrow \infty$ and the external addresses associated with  $\lambda$ in the exponential parameter plane. 
\end{abstract}
\maketitle
\section{Introduction}
The past few decades have seen significant advances in the dynamics of polynomials and entire functions. Taylor series expansions tell us that any entire function, such as $\exp(z)$, can be approximated by a sequence of polynomials. This often allows us to adapt the techniques of polynomial dynamics to the entire setting. With this philosophy, Devaney, Lyubich, Rempe, Schleicher and others have conducted broad studies of several
entire maps as limits of polynomials. Whether polynomial approximations can be tailor-made to satisfy specific dynamical properties of the limiting entire map is a broad and intriguing question. The authors of \cite{Hubbard_et_al}  considered this question in the setting of the exponential functions  $p_\lambda(z) = \lambda \exp(z)$ with $\lambda \in \C^*$, and showed that hyperbolic components and specific rays in the parameter plane of $\{p_\lambda\}$ are limits of the corresponding objects in the parameter space of unicritical polynomials $p_{d,\lambda }(z) = \lambda(1 + z/d)^d$ as the degree $d \rightarrow \infty$.  \cite{Kisaka} studied convergence of Julia sets of a sequence of entire functions when the sequence has a limit satisfying certain properties. In \cite{helena}, the author studied a type of kernel convergence of non-escaping hyperbolic components of general entire families. 

In this article, we study approximations of \textit{postsingularly finite} (psf) maps in the $p_\lambda$ family. An entire map is said to be postsingularly finite if the forward orbit of its set of singular values is finite. In the $p_\lambda$ family, these are maps for which the point $0$, the unique singular value, has a finite forward orbit under iteration. Since the singular values of  polynomials are exactly its critical values, we refer to psf polynomials as postcritically finite (pcf). Postcritically finite polynomials  admit various combinatorial descriptions, in terms of \textit{spiders} (see \cite{spideralgo}, \cite{teichmuller_theory_vol2}), \textit{Hubbard trees} (\cite{poirier}), etc. Spiders and \textit{homotopy Hubbard trees} have also been introduced for psf exponentials (see \cite{combi_classif_exp}, \cite{pfrang_rothgang_schleicher_2023} respectively). Polynomials of critically preperiodic type were completely classified in \cite{hubbard_bielefeld_fisher} using purely combinatorial techniques, and psf exponentials in \cite{combi_classif_exp}. Postsingularly finite maps are profitable to study, since they play an important role in understanding the parameter space of both unicritical polynomials and exponentials. The topology of the Mandelbrot set can largely be described by means of pcf quadratic polynomials, as shown in \cite{DH_Orsay} and \cite{fibers}. A similar study of the Multibrot sets was conducted in \cite{multibrot}. 

The machinery developed by William Thurston enables us to study topological models for postcritically finite rational maps. The question of when a topological pcf model is equivalent to a rational map is the subject of Thurston's characterization theorem (\cite{neg_char}), and other results such as Dylan Thurston's positive characterization of rational maps (\cite{pos_char}), the Levy-Bernstein theorem (\cite[Theorem~10.3.1]{teichmuller_theory_vol2}), etc, lean into this theme. In the polynomial setting, by a theorem of Bernstein, Lei, Levy and Rees (\cite[Theorem~10.3.9]{teichmuller_theory_vol2}),  it is known that a topological pcf polynomial is equivalent to a holomorphic one if and only if the topological model does not have a \textit{Levy cycle} (see Section~\ref{sec: thurston maps}). Thurston's theory for postsingularly finite transcendental maps has seen tremendous development in the last decade, kicked off by \cite{hss} for \textit{topological exponential maps}, and continued by \cite{Pfrang_Thesis} and \cite{dreadlock} for general transcendental maps. Work by Shemyakov (\cite{Shemyakov_Thesis}) proves that for Thurston maps of structurally finite type, the only obstructions are Levy cycles. 

In the setting of unicritical polynomials, the spider algorithm developed by \cite{spideralgo} is an implementation of the Thurston pullback operator in the case of critically periodic parameters and some critically preperiodic ones. This provides a link between the combinatorial data contained in spiders and the rigid geometry of the pcf polynomial. Similarly for exponentials, external addresses (see section~\ref{sec: exp dynamics}) play the role of angles in the finite degree case.  Our main result illustrates the utility of this relationship between combinatorics and geometry to prove the existence of a dynamical approximation of psf exponentials by pcf polynomials, and pinpoints these approximating polynomials by identifying angles of rays that land on them in parameter space. 

In concurrent work with Prochorov and Reinke, we show that we can approximate general psf entire maps by pcf polynomials locally uniformly. Our approach there is to translate combinatorial data contained in line complexes into an analytical approximation of the Thurston pullback operator (defined in Section~\ref{sec:thurston pullback}) of entire maps. While the more extensive literature on the combinatorics  unicriticals and exponentials is mined in order to poinpoint our approximating polynomials as landing points of specific parameter rays, our results here chalk out further questions about approximations of general psf entire maps.

\subsection{Results}
Given a Thurston map $f$ on the plane (defined in Section~\ref{sec: thurston maps}), we denote by $\Post{(f)}$ the postsingular set of $f$. When $f$ is a unicritical pcf polynomial, its critical value $0$ is either periodic or strictly preperiodic under iteration. 

We make extensive use of parameter rays to the space of unicriticals. Any map $p_{d,\lambda}$ with $\lambda \in \C^*$ is affine conjugate to a monic polynomial $z^d+c$, where $c$ is unique upto multiplication by a $(d-1)$th root of unity. Moreover, if $\lambda$ is a critically preperiodic parameter, any such $c$ is the landing point of a positive finite number of parameter rays at rational angles.

\noindent We state our main result below.
\begin{maintheorem}\label{thm:A}
Given a postsingularly finite exponential map $p_\lambda$ where $\lambda \in \C^*$, there exists an integer $D$ and postcritically finite polynomials  $p_{d,\lambda_d}$ for $d\geq D$ such that 
\begin{enumerate}
    \item $p_{d,\lambda_d} \longrightarrow p_\lambda$ locally uniformly on $\C$ as $d \rightarrow \infty$
    \item $p_\lambda | \Post{(p_\lambda)}$ is conjugate to $p_{d,\lambda_d}| \Post{(p_{d,\lambda_d})}$
\end{enumerate}
Furthermore, there exists a polynomial $q \in \Z[X]$, integers $\ell,k \geq 1$ with $\deg q  \leq \ell+k-2$ depending only on $\lambda$, and a sequence of monic  polynomials $z^d+c_d$ conjugate to $p_{d,\lambda_d}$,  such that  $c_d$ is the landing point of the parameter ray at angle $\theta_d$, with $(d-1)\theta_d  = \frac{(d-1)q(d)}{d^\ell(d^k-1)}$ for all $d$. 
\end{maintheorem}
Each Thurston map $f: (S^2 \setminus \{\infty\} , \Post(f)) \righttoleftarrow$ gives rise to an  operator  $\sigma_f: \Teich(S^2, \Post(f)) \righttoleftarrow$, called the Thurston pullback operator of $f$ (see Section~\ref{sec:thurston pullback}), which has a fixed point $[\phi]$ in  $\Teich(S^2, \Post(f))$ if and only if $f$ is realized as (ie, is Thurston equivalent to) a psf entire function $g$. If $f$ is realized, the function $g = \phi' \circ f \circ \psi$ for some $\phi', \psi \in [\phi]$. 

In the case of psf unicriticals and exponentials, such topological models can be constructed using the underlying spiders of these functions. Given a psf polynomial $p_{d,\lambda}$, choose any $z^d+c$ conjugate to $p_{d,\lambda}$ and an angle $\theta$ such that the corresponding parameter ray lands at $c$. We recall in Section~\ref{sec:spiders} the degree$-d$ spider map $\mathcal{F}_{d,\theta}: S_d(\theta) \righttoleftarrow$. A certain quotient $[\mathcal{F}_{d,\theta}]: [S_d(\theta)] \righttoleftarrow $ obtained by identifying points with the same itinerary can be extended to a branched cover of $S^2$ that is Thurston equivalent to a pcf $p_{d, \lambda}$. We introduce a \textit{spider space} $\mathcal{T}^d_{\theta}$ (a type of Teichm\"{u}ller space) associated with $[S_d(\theta)]$ and an operator $\sigma^d_{\theta}$ on $\mathcal{T}^d_\theta$ (which is highly related to the spider operator in \cite{spideralgo}). This operator is conjugate to the Thurston pullback operator of $[\mathcal{F}_{d,\theta}]$, and has a unique fixed point $[\phi]$ in $\mathcal{T}^d_\theta$, using which we can recover $p_{d,\lambda}$ as $\phi' \circ[\mathcal{F}_{d,\theta}] \circ \phi $. 

In direct analogy, there exists a notion of dynamic/parameter rays to the exponential family $p_\lambda$. Each ray in the exponential setting is labelled by a sequence $\underline{s} \in \Z^{\N}$ called an external address. It is known that each psf exponential parameter $\lambda$ is the landing point of a positive finite number of parameter rays with strictly preperiodic external addresses. Conversely, the parameter ray at every strictly preperiodic external address lands on an psf exponential parameter (see Section~\ref{sec: exp dynamics}). Thus, given psf $p_\lambda$ where $\lambda$ has external address $\underline{s}$, following \cite{combi_classif_exp} we construct a planar graph $\Gamma_{\underline{s}}$ called an exponential spider, and an associated graph map $\mathcal{F}_{\underline{s}}:\Gamma_{\underline{s}} \righttoleftarrow$. A quotient dynamical system $[\mathcal{F}_{\underline{s}}]:[\Gamma_{\underline{s}}] \righttoleftarrow$ can be extended to topological psf map of the plane Thurston equivalent to $p_\lambda$. We construct a Teichm\"{u}ller space  $\mathcal{T}_{\underline{s}}$ with an associated operator $\sigma_{\underline{s}}$ that is conjugate to $\sigma_{[\mathcal{F}_{\underline{s}}]}$, and has a fixed point $[\phi] $ in $\mathcal{T}_{\underline{s}}$. As before, we have $p_\lambda = \phi' \circ [\mathcal{F}_{\underline{s}}] \circ \psi$ for some $\phi',\psi \in [\phi]$.

What we have detailed is the iteration of Thurston pullback operators using spiders. We uncover relationships between spiders $\Gamma_{\underline{s}}$ for preperiodic $\underline{s}$, and spiders $S_d(\theta)$ as $d \rightarrow \infty$.  Our first step is to show that any degree$-d$ spider is realised as a degree$-d+1$ spider. In other words, the spiders $S_d(\theta)$ of rational angles $\theta$ are \textit{persistent} as the degree $d \rightarrow \infty$, . 
\begin{lemma}\label{lemma:injective_maps_of_angles}
Given a degree $d\geq 2$, there exist $d$ injective order preserving maps $Z_j : \Q/\Z \longrightarrow \Q/\Z$ such that for every $\theta \in \Q/\Z$, $S_d(\theta)$ and $S_{d+1}(Z_j(\theta))$ are isomorphic.
\end{lemma}
This persistence of spiders underlies a deeper relationship between the space of pcf unicriticals in one degree to the next. Let $\mathcal{P}_d$ denote the collection of $\lambda $ in $\C^*$ such that $p_{d,\lambda}$ is pcf. We induce a partial order on $\mathcal{P}_d$ that records when one parameter is in the wake of another (for the exact definition, see Section~\ref{sec: combi embeddings defn}). A combinatorial embedding $\Phi$ from $\mathcal{P}_d$ to $\mathcal{P}_{d+1}$ is an injective, order-preserving map that preserves satellite relationships, such that for each $\lambda \in \mathcal{P}_d$, $p_{d,\lambda} | \Post(p_{d,\lambda})$ is conjugate to $p_{d+1,\Phi(\lambda)}|\Post(p_{d+1,\Phi(\lambda)})$ by a bijection mapping $0$ to $0$ (see Definition~\ref{defn:combi_embedding}). We use Lemma~\ref{lemma:injective_maps_of_angles} to show the following.
\begin{lemma}\label{lemma:combi_embeddings}Given a degree $d\geq 2$, there exist $d$ combinatorial embeddings $\mathcal{E}_j: \mathcal{P}_d \longrightarrow \mathcal{P}_{d+1}$, $j = 0,1,...,d-1$.
\end{lemma}

 Next, we show that spiders corresponding to preperiodic external addresses are realised as degree $d$ spiders for sufficiently large $d$. 
\begin{lemma}\label{lemma: external address correspondence}
Given an external address $\underline{s}$ with preperiod $\ell\geq 1$, period $k\geq 1$ such that $s_1 = 0$, there exists a degree $D\geq 2$, $j \in \{0,1,...,D-1\}$ and angles $\theta_d$ for  $d\geq D$ such that
\begin{enumerate}
    \item $\theta_{d+1} = Z_j(\theta_d)$ for all $d\geq D$.
    \item The spiders $S_d(\theta_d)$ and $\Gamma_{\underline{s}}$ are isomorphic for all $d\geq D$
\end{enumerate}
\end{lemma}
As a next step, we prove an analytic relationship between $\sigma_{\underline{s}}$ and $\sigma^d_{\theta_d}$, when $\underline{s}$ and $\underline{\theta_d}$ are as above.  We show that $\sigma_{\underline{s}}$ can be approximated by a sequence of operators conjugate to $\sigma^d_{\theta_d}$, and the angles $(d-1)\theta_d$ satisfy a certain stability property . 

\begin{lemma}\label{lemma: convergence of spider maps}
    Let $\underline{s}$ be a preperiodic external address with preperiod $\ell \geq 1$ and period $k \geq 1$, and let $\mathcal{T}_{\underline{s}}$ be the spider space of $\underline{s}$. There exists a polynomial $q \in \Z[x]$ with $\deg q \leq \ell+k-2$,  a sequence of angles $\theta_d$ and a sequence of operators $\sigma^d_{\underline{s}}: \mathcal{T}_{\underline{s}} \longrightarrow \mathcal{T}_{\underline{s}}$ such that 
    \begin{enumerate}
    \item $(d-1)\theta_d = \frac{(d-1)q(d)}{d^\ell(d^k-1)}$ for all $d$,
        \item $\sigma^d_{\underline{s}}: \mathcal{T}_{\underline{s}} \righttoleftarrow$ is conjugate to $\sigma^d_{\theta_d}: \mathcal{T}^d_{\theta_d} \righttoleftarrow$, and 
        \item $\sigma^d_{\underline{s}} \longrightarrow \sigma_{\underline{s}}$  uniformly on compact sets of $\mathcal{T}_{\underline{s}}$.
    \end{enumerate}
\end{lemma}
These lemmas, along with the theory of quasiconformal maps,  form the major ingredients for proving Theorem~\ref{thm:A}. 
\subsection{Organization of the paper}In Section~\ref{sec:background} we introduce background material on Thurston theory, dynamics of psf unicriticals and exponential maps, spiders, orbit portraits and kneading sequences. We introduce the partial order on $\mathcal{P}_d$ that makes it meaningful to define combinatorial embeddings. In Section~\ref{sec:proofoflemma1.1}, we  construct the maps $Z_j$ and show that they satisfy the properties listed in Lemma~\ref{lemma:injective_maps_of_angles}. We additionally study how combinatorial properties like kneading sequences change under $Z_j$. In Section~\ref{sec:proofoflemma1.2}, we use the maps $Z_j$ to construct the embeddings $\mathcal{E}_j$ and prove Lemma~\ref{lemma:combi_embeddings}. We also provide a combinatorial description of the image of $\mathcal{P}_d$ under each $\mathcal{E}_j$. Sections~\ref{sec:proofoflemma1.3} and \ref{sec:proofoflemma1.4} contain the proof of Lemmas~\ref{lemma: external address correspondence} and \ref{lemma: convergence of spider maps} respectively. We end by proving Theorem~\ref{thm:A} in Section~\ref{sec:maintheorem}.
\subsection{Acknowledgements} The author is deeply grateful to John Hubbard for posing the question of dynamical approximation of exponentials and the combinatorial properties of these approximations, and to Sarah Koch for  ingenious ideas and continued guidance. Sincere thanks to Dierk Schleicher for introducing the broader question of approximating general postsingularly finite entire functions as a natural continuation, and for supporting the author throughout this work. The author is indebted to Nikolai Prochorov and Bernhard Reinke for the many involved discussions during our work on the general approximation question. The author was also greatly benefitted by discussions with Giulio Tiozzo, Dylan Thurston, Lasse Rempe, Caroline Davis, Alex Kapiamba and Schinella D'Souza.

The author was supported by the National Science Foundation under Grant No.\ DMS-1928930
while in residence at the Mathematical Sciences Research Institute in Berkeley, California, during
the Spring 2022 semester.
\section{Background}\label{sec:background}
Throughout this section, fix a degree $d\geq 2$. 
\subsection{Operations on angles}
All angles in this paper are taken to be elements of $\R/\Z$. Given distinct angles $\alpha,\beta$, $\R/\Z \setminus \{\alpha,\beta\}$ consists of two connected components or \textit{arcs}. The length of the shorter arc is denoted $d(\alpha,\beta)$. We take the linear order on $\R/\Z$ induced by that on $[0,1)$. The map $\mu_d: \R/\Z \righttoleftarrow$ is defined as $\mu_d(x) = dx$, and we let $\mathscr{O}_d(\theta) = \{d^{n-1}\theta: n\geq 1\}$. 

Every rational angle is preperiodic under $\mu_d$ with preperiod $\ell\geq 0$ and eventual period $k\geq 1$ (we will often drop the word `eventual'). 
\subsubsection{Itineraries and kneading data}
Given $\theta \in \R/\Z$, for $j = 0,1,...,d-1$, we define the $j$th static sector with respect to $\theta $ as the interval $(\frac{\theta+j}{d},\frac{\theta+j+1}{d}) \subset \R/\Z$, and denote it $T^{stat}_{d,j}(\theta)$. Note that $\bigsqcup_{j=0}^{d-1}T^{stat}_{d,j}(\theta) = \R/\Z \setminus \mu_d^{-1}(\theta)$. 

Now suppose $\theta \in T^{stat}_{d,i}(\theta)$. For $j = 0,1,...,d-1$, we define the $j$th dynamic sector with respect to $\theta$ to be $T^{stat}_{d,j+i}(\theta)$, and denote it $T^{dyn}_{d,j}(\theta)$. Dynamic sectors are well-defined if and only if $\theta \not \in \mu_d^{-1}(\theta)$. 

For any angle $t \in \R/\Z$, the itinerary of $t$ with respect to $\theta$, denoted $\Sigma_{d,\theta}(t)$, is the sequence $\nu_1\nu_2\nu_3.... \in \{0,1,...,d,*\}^{\N}$ where 
\begin{align*}
    \nu_n& = \begin{cases}
        j & d^{n-1}t \in T^{dyn}_{d,j}(\theta)\\
        * & d^{n-1}t \in \mu_d^{-1}(\theta)
    \end{cases}
\end{align*}
An itinerary is called $*-$periodic if it is periodic under the shift map $\nu_1\nu_2\nu_3\nu_4.... \mapsto \nu_2\nu_3\nu_4....$ with period $k$, and the $*$'s occur exactly at indices $k,2k,3k, $ etc. For an angle $\theta$, the itinerary $\Sigma_{d,\theta}(\theta)$ is called the kneading sequence of $\theta$. Kneading sequences of rational angles are  preperiodic or $*-$periodic under the shift map.

\subsection{Thurston maps on the plane}
The theory developed by William Thurston to study postsingularly finite entire functions develops topological models for these functions and asks when such models are `equivalent' in a suitable sense to holomorphic ones. We give the necessary background on Thurston theory in this section. 

Let $f: S^2 \setminus \{\infty\} \righttoleftarrow$  a topological branched cover. If $\deg f$ is finite, then $f$ is said to be a \textit{topological polynomial}. A value $a \in S^2 \setminus \{\infty\} $ is said to be an \textit{asymptotic value} for $f$ if there exists an arc $\gamma: [0,\infty) \longrightarrow t$ such that $\lim_{t \rightarrow \infty} \gamma(t) = \infty$ and $\lim_{t \rightarrow \infty} f(\gamma(t)) = a$. 
The critical and asymptotic values of $f$ are called its \textit{singular values}, and we denote the set of singular values by $S(f)$. The \textit{postsingular set} $\Post{(f)}$ of $f$ is the forward orbit of $S(f)$. If $\Post{(f)}$ is finite, we say that $f$ is \textit{postsingularly finite} (psf). Since topological polynomials do not have asymptotic values, a psf topological polynomial is said to be \textit{postcritically finite} (pcf).

\begin{definition}\label{defn:isom}
    Given any complex structure $\mu$ on $S^2 \setminus \{\infty\}$ such that $(S^2 \setminus \{\infty\},\mu)$ is isomorphic to $\C$, the Riemann surface $(S^2 \setminus \{\infty\}, f^*\mu)$ is isomorphic either to $\C$ or $\D$. If $(S^2 \setminus \{\infty\}, f^*\mu)$ is isomorphic to $\C$ for every choice of $\mu$, we call $f$ a \textit{topological entire} function.
\end{definition}
 We note that any topological polynomial is topologically entire.
 \subsubsection{Thurston maps}\label{sec: thurston maps}
 \begin{definition}
A psf topological entire function $f: (S^2 \setminus {\infty} , A) \righttoleftarrow$ where $\Post{(f)} \subset A$ and $A$ is a finite set, is called a \textit{Thurston map}. Unless specified, we assume that $A = \Post{(f)}$.      
 \end{definition}
\begin{definition}
    Thurston maps $f: (S^2 \setminus {\infty} , A) \righttoleftarrow$ and $g: (S^2 \setminus {\infty} , B) \righttoleftarrow$ are said to be \textit{Thurston equivalent} if there exist orientation-preserving homeomorphisms $\varphi_0,\varphi_1: S^2 \setminus \{\infty\} \righttoleftarrow$ such that 
\begin{itemize}
    \item $g \circ \varphi_0 = \varphi_1 \circ f$
    \item $\varphi_0|A \equiv \varphi_1|A$ and $\varphi_0(A) = \varphi_1(A) = B$
    \item $\varphi_0 \sim \varphi_1$ rel $A$
\end{itemize}

\end{definition}
The map $f$ is said to be \textit{realized} if $f$ is Thurston equivalent to some psf entire function $g$; otherwise it is \textit{obstructed}.

For topological polynomials, obstructions to being realized are well-understood. Let $\Gamma = \{\gamma_0 = \gamma_r,\gamma_1,\gamma_2,...,\gamma_{r-1}\}$ be a collection of essential simple closed curves $\gamma_i \subset S^2 \setminus \Post{(f)}$, which are pairwise disjoint and belong to distinct homotopy classes relative to $\Post{(f)}$. Such a collection is called a \textit{Levy cycle} for $f$ if for every $i \in \{0,1,...,r\}$, at least one connected component $\eta$ of $f^{-1}(\gamma_i)$ is isotopic to $\gamma_{i-1}$ relative to $\Post{(f)}$, and $f|\eta: \eta \longrightarrow \gamma_i$ has degree $1$.

The following result was proven by Bernstein, Lei, Levy and Rees. 
\begin{theorem}
    [{\cite[Theorem~10.3.9]{teichmuller_theory_vol2}}]
    If $f$ is a topological polynomial, it is realized if and only if it does not have a Levy cycle. 
\end{theorem}
Recent work by Shemyakov (\cite{Shemyakov_Thesis}) extends this result to the structurally finite topologically entire family. 

\subsubsection{Topological exponential functions}
Note that we will be working only in the context of unicritical polynomials and exponential functions. 
\begin{definition}
A universal cover $f: S^2 \setminus \{\infty\} \longrightarrow S^2 \setminus \{0,\infty\}$ is called a \textit{topological exponential function}. 
\end{definition}
For a function $f$ as above, there is a unique asymptotic value at $0$, and we have $S(f) = \{0\}$. As proved in \cite{hss}, if $f$ is as above, it is topologically entire: given a complex structure on $S^2 \setminus \{\infty\}$ such that $(S^2 \setminus \{\infty\},\mu)$ is isomorphic to $\C$, let $\phi: S^2 \setminus \{\infty\} \longrightarrow \C$ be an integrating map for $\mu$ with $\phi(0) = 0$. The holomorphic map $\phi \circ f: (S^2 \setminus \{\infty\},f^*\mu) \longrightarrow \C^*$ is a universal cover. Since there is no universal cover from $\D$ onto $\C^*$, the Riemann surface $(S^2 \setminus \{\infty\},f^*\mu) $ is isomorphic to $\C$.

\subsubsection{Teichm\"{u}ller spaces of punctured spheres}\label{pullbackoperatordefn}
Let $f:(S^2,\Post{(f)}) \righttoleftarrow$ be a Thurston map. If we want to find a psf entire map $g$ that is Thurston equivalent to $f$, according to Thurston theory we start by looking at complex structures induced by all possible maps $\phi: S^2 \setminus \{\infty\} \longrightarrow \C$ and try to guess a suitable $\psi,g$ such that $g = \phi \circ f \circ \psi^{-1}$. We describe this process in this section and the next.

\begin{definition}
    Given a finite set $A \subset S^2$ with $|A| \geq 3$, the \textit{Teichm\"{u}ller space} $\Teich(S^2,A)$ is defined as 
\begin{align*}
    \Teich(S^2,A)& =\{\text{Orientation-preserving homeomorphisms }\phi: S^2 \longrightarrow \hat{\C}\}/\sim
\end{align*}
where $\phi \sim \psi$ if, for some  $M \in \text{Aut}(\hat{\C})$, $\phi$ is isotopic to $M \circ \psi$ rel $A$. 
\end{definition}
\noindent $\Teich(S^2,A)$ is a contractible complex manifold of dimension $|A| - 3$. The \textit{Teichm\"{u}ller metric} on $\Teich(S^2,A)$ is defined as follows: given $[\phi], [\psi] \in \Teich(S^2,A)$,
\begin{align*}
    d_{T}([\phi], [\psi]) & = \inf_{q.c. \psi' \sim \phi \circ \psi^{-1}}\ln K(\psi)
\end{align*}
where $\psi' \sim \phi \circ \psi^{-1}$ means $\psi'$ is isotopic to $\phi \circ \psi^{-1}$ rel $A$, and $K(\psi)$ is the complex dilitation of $\phi\circ \psi^{-1}$.  
\subsubsection{The Thurston pullback operator}\label{sec:thurston pullback}
\begin{definition}
    Let $f:(S^2 \setminus \{\infty\}, \Post{(f)}) \righttoleftarrow$ be a Thurston map. The \textit{Thurston pullback operator} $\sigma_f:\Teich(S^2,\Post{(f)} \cup \{\infty\}) \longrightarrow \Teich(S^2,\Post{(f)} \cup \{\infty\})$ is defined as follows:

Given $[\phi] \in \Teich(S^2,\Post{(f)} \cup \{\infty\})$, where a representative $\phi$ is chosen such that $\phi(\infty) = \infty$, $\sigma_f([\phi]) = [\psi]$, where $\psi$ completes the diagram
\[
\begin{tikzcd}
\Big(S^2 \setminus \{\infty\},\Post{(f)} \Big) \arrow[r,"\psi"] \arrow[d,"f"]&\Big(\C,\psi(\Post{(f)})\Big) \arrow[d,"g"]\\
\Big(S^2 \setminus \{\infty\},\Post{(f)} \Big) \arrow[r,"\phi"] & \Big(\C,\phi(\Post{(f)})\Big)
\end{tikzcd}
\]
for some entire function $g$.

\end{definition}
\begin{itemize}
    \item It is easy to show that $f$ is realized if and only if $\sigma_f$ has a fixed point in $\Teich(S^2,\Post{(f)}\cup \{\infty\})$. 
\item As shown in \cite{hss} for topological exponential functions, \cite{spideralgo}, $\sigma_f$ is weakly contracting with respect to the Teichm\"{u}ller metric on $\Teich(S^2,\Post{(f)}\cup \{\infty\})$. That is, given $[\phi],[\psi] \in \Teich(S^2,\Post{(f)}\cup \{\infty\})$, we have 
\begin{align*}
    d_T(\sigma_f([\phi]),\sigma_f([\psi])) < d_T([\phi],[\psi])
\end{align*}
Thus if $\sigma_f$ has a fixed point, it is unique. Consequently, if $f$ is Thurston equivalent to an entire map $g$, then $g$ is unique upto conjugation by an affine  transformation.
\item If $f$ is a topological entire function, given any compact subset $K$ of $\Teich(S^2,\Post{(f)}\cup \{\infty\})$, $\sigma_f$ is strongly contracting on $K$: that is, there exists a constant $c_K \in [0,1)$ such that for all $[\phi],[\psi] \in K$,
\begin{align*}
    d_T(\sigma_f([\phi]),\sigma_f([\psi])) \leq c_K d_T([\phi],[\psi])
\end{align*}
A proof is given in \cite{hss}. 
\end{itemize}

\subsection{Dynamics of unicritical polynomials}\label{sec: unicritical dynamics}
Given $c \in \C$, we let $f_{d,c}$ denote the polynomial $z^d+c$. On some maximal neighborhood $U_c$ of $\infty$, there exists a unique conformal map $\phi_{d,c}: U_c \longrightarrow \{|z|<r\}$ for some $0<r<1$ satisfying 
\begin{align*}
    \phi_{d,c}\circ f_{d,c}(z) &= \phi_{d,c}(z)^d \hspace{5pt} \forall z \in \hat{\C} \setminus K(f_{d,c})\\
    \lim_{z \mapsto \infty}\frac{\phi_{d,c}(z)}{z} & = 1
\end{align*}
This map is referred to as the \textit{B\"{o}ttcher chart} of $f_{d,c}$. If the filled Julia set $K(f_{d,c})$ is connected (equivalently, if the $f_{d,c}$-orbit of the critical point $0$ is bounded), then $U_c = \hat{\C} \setminus K(f_{d,c})$. 

The \textit{dynamic ray at angle $\theta \in \R/\Z$} is the preimage of the set $\{re^{2\pi i \theta}: 1\leq r \leq \infty\}$. If $\lim_{r \rightarrow 1^+}\phi_{d,c}^{-1}(re^{2 \pi i \theta})$ exists, we say that the ray \textit{lands}. It is known that dynamic rays at rational angles always land. 

If $K(f_{d,c})$ is locally connnected, $\phi_{d,c}^{-1}$ extends continuously to the boundary  $\partial \D$ and maps onto the Julia set $J(f_{d,c}) = \partial K(f_{d,c})$. This boundary extension  is called the Carath\'{e}odory loop of $f_{d,c}$. 
\subsection{Unicritical non-escaping loci}
The set of $c \in \C$ for which the $f_{d,c}-$ orbit of $0$ is bounded, is the \textit{non-escaping locus} of the polynomials $f_{d,c}$, commonly known as the degree $-d$ Multibrot set. We denote this set by $\mathcal{M}_d$. We recall some known facts about $\mathcal{M}_d$ here. Proofs can be found in \cite[Chapters~9, 10]{teichmuller_theory_vol2}, and \cite{multibrot}.
\begin{itemize}
    \item $\mathcal{M}_d$ is connected and compact, and the map $\Phi_d: \hat{\C} \setminus \mathcal{M}_d \longrightarrow \hat{\C} \setminus \overline{\D}$ given by $\Phi_d(c) = \varphi_{d,c}(c)$, where $\varphi_{d,c}$ is the B\"{o}ttcher chart of $f_{d,c}$, is a conformal isomorphism. This was proven first for $d=2$ by Douady and Hubbard in \cite{Mandelbrotconnected}, and their proof generalizes to higher degrees.

    Given $\theta \in \R/\Z$, the parameter ray at angle $\theta$, denoted $R_d(\theta)$, is the preimage under $\Phi_d$ of the set $\{re^{2\pi i \theta}:1 < r\leq \infty\}$. Landing of parameter rays is defined analogously.  If  $\theta \in \Q/\Z$, $R_d(\theta)$ lands.
    \item Given any hyperbolic component $U \subset \mathcal{M}_d$ (for a definition, see,  \cite[Chapter 3]{orsaynotes} or \cite[Chapter 9] {teichmuller_theory_vol2}), the multiplier map $\rho_U: U \longrightarrow \D$ is a $(d-1)$ - sheeted ramified covering branched over $0$. The unique critical point $c_0$ of $\rho_U$ is called the center of $U$, and $f_{d,c_0}$ has a super-attracting cycle of exact period $k$. For $c \in U$, $f_{d,c}$ has a unique attracting cycle of exact period $k$. $\rho_U$ extends to a continuous map  $\partial U \longrightarrow \partial \D$, and the fiber $\rho_U^{-1}(1)$ consists of $d-1$ parabolic parameters  $c_1,c_2,...,c_{d-1}$ on $\partial U$. 
    
    Of these, there exists a unique parameter, say $c_1$, which is the landing point of exactly two rays: $R_d(\theta)$ and $R_d(\theta')$, where $\theta,\theta'$ are periodic under $\mu_d$ with exact period $k$. This point is called the root of $U$, and the angles $\theta,\theta'$ are said to form a companion pair. The points $c_2,....,c_{d-1}$ are called the co-roots of $U$ and each $c_i$ is the landing point of exactly one ray $R_d(\theta_i)$, where $\theta_i$ has exact period $k$ under $\mu_d$. In the dynamical plane of $f_{d,c_0}$, the dynamic rays at angles $\theta',\theta,\theta_2,...,
    \theta_{d-1}$ all land on the Fatou component $U_0$ containing $c_0$, and $\theta,\theta'$ land at a unique point $z_0 \in \partial U_0$, called its root. The dynamic rays at angles $\theta,\theta'$ separate $c_0$ from the other points in the postcritical set.
    \item 
Let $U$ be a hyperbolic component of $\mathcal{M}_d$, with center $c_0$. Suppose the parameter rays $R_d(\theta),R_d(\theta')$ with $\theta<\theta'$ land at the root of $U$. Then $R_d(\theta) \cup R_d(\theta')$ split $\mathcal{M}_d \setminus U$ into two components. The component not containing $0$ is called the wake of $U$, and denoted $\mathcal{W}(U)$ (we will also refer to this as the wake $c_0$). Furthermore, suppose $\theta_1 < \theta_2 < ....\theta_{d-2}$ are the angles that land at the co-roots of $U$. Then the subsets of $\mathcal{W}(U)$ bound by a pair of rays of the form $(R_d(\theta)_j , R_d(\theta_{j+1}))$, $(R_d(\theta), R_d(\theta_1))$ or $(R_d(\theta_{d-2}), R_d(\theta'))$ are called sub-wakes of $U$.  

    \item If $\theta \in \Q/\Z$ is $k-$periodic under $\mu_d$, $R_d(\theta)$ lands on the root or co-root of a hyperbolic component of period $k$. If $\theta$ is preperiodic under $\mu_d$ with pre-period $\ell\geq 1$ and eventual period $k$, then $R_d(\theta)$ lands at a Misiurewicz parameter $c$ whose critical value has preperiod $\ell$ and period dividing $k$. 
    \item    We call $\theta \in \Q/\Z$ an angular coordinate for $c \in \mathcal{M}_d$ if 
\begin{enumerate}
    \item $c$ is Misiurewicz, and $R_d(\theta)$ lands on $c$, or 
     \item $c$ is critically periodic, and $R_d(\theta)$ lands on the root or a co-root of the hyperbolic component containing $c$.
\end{enumerate}
Given a pcf parameter $c$, we let $\Omega_d(c)$ denote its set of angular coordinates. For $\omega = e^{\frac{2 \pi i}{d-1}}$,  
\begin{align*}
        \Omega_d(\omega c) & = \Omega_d(c) + \frac{1}{d-1}
\end{align*}
\item If $f_{d,c}$ is pcf, with $\theta \in \Omega_d(c)$, its filled Julia set is locally connected. Let $\gamma$ the Carath\'{e}odory loop of $f_{d,c}$. For angles $t,t' \in \R/\Z$, it is known that $\gamma(t) = \gamma(t')$ if and only if $\Sigma_{d,\theta}(t) = \Sigma_{d,\theta}(t')$. 
    \end{itemize}    

\subsection{The quotients $\Lambda_d$}
For all $c\in \C$, $f_{d,c}$ is affine conjugate to $f_{d,\omega c}$, where $\omega = e^{\frac{2 \pi i}{d-1}}$. Moreover, for $c \neq 0$, $f_{d,c}$ is affine conjugate to $p_{d,\lambda}(z) = \lambda(1+\frac{z}{d})^d$, where $\lambda = dc^{d-1}$, and $p_{d,\lambda}$ is affine conjugate to $p_{n,\mu}$ if and only if $\lambda = \mu$. 

We define $\Lambda_d$ as the image of $\mathcal{M}_d \setminus \{0\}$ under the map $c \mapsto dc^{d-1}$. Equivalently, $\Lambda_d$ is the set of $\lambda \in \C^*$ such that $p_{d,\lambda}$ has connected filled Julia set. 

Given $\lambda =dc^{d-1}$, we refer to the parameters $c,\omega c,\omega^2c,...,\omega^{d-2}c$ as the monic representatives for $\lambda$, and denote this set $M_d(\lambda)$. We call $\theta$ an \textit{angular coordinate} for $\lambda$ if $\theta \in \Omega_d(c)$ for some $c \in M_d(\lambda)$. The set of angular coordinates is denoted $\Theta_d(\lambda)$.
\begin{align*}
    \Theta_d(\lambda) & = \bigcup_{n = 0}^{d-2} \Omega_d(\omega^nc)
\end{align*}
We shall denote by $\mathcal{P}_d$ the set of pcf parameters $\lambda$ for which $p_{d,\lambda}$ is pcf. This consists of all Misiurewicz parameters, and all critically periodic parameters with period $\geq 2$. 
\subsection{Combinatorics of pcf polynomials}
\subsubsection{Orbit portraits}
Given any repelling cycle $\{z_1,...,z_r\}$ of a polynomial $p$, the orbit portrait associated with this cycle is the collection $\{\mathscr{A}_1,\mathscr{A}_2,...,\mathscr{A}_r\}$, where $\mathscr{A}_i$ is the set of angles $\theta \in \R/\Z$ such that the dynamic ray at angle $\theta$ in the plane of $p$ lands at $z_i$. 
\begin{definition}\label{defn:formal_portrait}
   A collection  $\mathscr{A} = \{\mathscr{A}_1,\mathscr{A}_2,...,\mathscr{A}_r\}$ is called a formal degree$-d $ orbit portrait if 
    \begin{enumerate}
        \item Each $\mathscr{A}_i$ is a non-empty finite subset of $\Q/\Z$. 
        \item  $\mu_{d}$ maps each $\mathscr{A}_i$ bijectively onto $\mathscr{A}_{i+1}$ and preserves the cyclic
order of the angles.
\item Each $\mathscr{A}_i$ is contained in some arc of length less than $1/d$ in $\R/\Z$.
\item Each $\alpha \in
\cup_{i=1}^r\mathscr{A}_i$ is periodic under $\mu_d$ and all such $\alpha$'s have a common
period $rp$ for some $p \geq 1$. 
\item For every $n$, define $\mathscr{A}_{n,i}$ := $\mathscr{A}_n + \frac{i}{d+1}$. For all pairs $(m,i) \neq (n,j)$, $\mathscr{A}_{m,i}$ and $\mathscr{A}_{n,j}$ are unlinked. 
    \end{enumerate}

\end{definition}

    Given any formal portrait $\mathscr{A}$, there exists $i$ and distinct angles $\alpha,\beta \in \mathscr{A}_i$ such that $d(\alpha,\beta)$ is uniquely minimal among all arc lengths $d(\alpha',\beta')$ where $\alpha'$ and $ ,\beta'$ are distinct angles in some $ \mathscr{A}_j$. The pair $(\alpha,\beta)$ is called the characteristic pair. $\mathscr{A}$ can be reconstructed from $(\alpha,\beta)$, and we call $\mathscr{A}$ the $d-$orbit portrait generated by $(\alpha,\beta)$. 
    
 \cite[Theorem~2.12]{multibrot}, states that every formal degree$-d$ portrait is realized as the orbit portrait associated with the cycle $(z_1,....,z_r)$ of a pcf polynomial $f_{d,c}$ whose critical value has period $rp$,  with each $z_i$ the root of the Fatou component containing $f_{d,c}^{\circ (i-1)}(c)$ ($z_{r+1} = z_1$, $z_{r+2} = z_2$, etc). 

 The authors of \cite{hubbard_bielefeld_fisher} carry out a combinatorial classification of critically preperiodic polynomials of a given degree polynomials of a given degree, in terms of angles landing on the orbit of the critical values. Throughout this paper we will use several properties of dynamic rays.
 \subsection{Dynamics of exponential maps}\label{sec: exp dynamics}
The set of $\lambda \in \C^*$ for which the $p_{\lambda}-$orbit of $0$ is bounded is the \textit{non-escaping locus} of the exponential family $\{p_\lambda\}$, denoted $\Lambda$. Any $\lambda \not \in \Lambda$ is called an \textit{escaping parameter}.
\subsubsection{Dynamic and parameter rays}
Given $\lambda \in \C^*$, choose $c = \ln  \lambda$ such that $\Im c \in [-\pi,\pi]$. Let $U_j^{stat}(\lambda) = \{z: (2j-1)\pi - \Im c < \Im z < (2j+1)\pi + \Im c\}$. Note that this is a connected component of $p_{\lambda}^{-1}(\C \setminus \R_{\leq 0})$, and that the map $p_\lambda | U_j^{stat}(\lambda) : U_j^{stat}(\lambda) \longrightarrow \C \setminus \R_{\leq 0}$ is a conformal isomorphism. The collection $\{U_j^{stat}(\lambda)\}$ partitions the plane, and is called the static partition with respect to $\lambda$. 

\begin{definition}
    We call a sequence $\underline{s} \in \Z^{\N}$ an \textit{external address}. Let $\sigma$ be the left shift map on $\Z^{\N}$. For any $z \in \C$ with $p_{\lambda}^{\circ n}(z) \not \in \R_{\leq 0}$ for all $n$, the external address of $z$ is the sequence $s_1s_2...$ with $p_{\lambda}^{\circ (n-1)}(z) \in U_{s_{n}}^{stat}(\lambda)$ for all $n$.

    Let $<$ denote the standard lexicographic order on $\Z^{\N}$: $\underline{s} < \underline{t}$ if at the first index $n$ where $s_n \neq t_n$, we have $s_n < t_n$. Additionally, let $<<$ denote the cylindrical order: $\underline{s} << \underline{t}$ if one of the following is true:
    \begin{itemize}
        \item $\underline{s} < \underline{t}$, or
        \vspace{5pt}
        \item $
        \underline{t} < \overline{0} < \underline{s}$
    \end{itemize}
\end{definition}
\begin{definition}
Given an address $\underline{s}$, for $r \in \Z$ let $T_r(\underline{s})$ denote the interval $(r\underline{s},(r+1)\underline{s}) = \{ \overline{t}: r\underline{s} < \overline{t} <(r+1)\underline{s})  \}$. Given $\underline{t} \in \Z^{\N}$, the itinerary of $\underline{t}$ with respect to $\underline{s}$, denoted $\Sigma_{\underline{s}}(\underline{t})$, is a sequence $\nu_1\nu_2\nu_3....$ where 
\begin{align*}
    \nu_n& = \begin{cases}
        r & \sigma^{\circ (n-1)}(\underline{t}) \in T_r(\underline{s})\\
        * &\sigma^{\circ (n-1)}(\underline{t}) = r\underline{s} \text{ for some }r \in \Z 
    \end{cases}
\end{align*}
\end{definition}
\begin{definition}
  A sequence $\underline{s} \in \Z^{\N}$ is \textit{bounded} if there exists a constant $C>0$ such that $|s_n| \leq C$ for all $n$. 
  
  Let $F(t) = e^t-1$.  An address $\underline{s}$ is said to be \textit{exponentially bounded} if there exist constants $A\geq 1$, $x>0$ such that $|s_n| \leq A|F^{\circ (n-1)}(x)|$ for all $n\geq 1$. 
  
\end{definition}

Fix $\lambda \in \C^*$. The following theorems illustrate the behaviour of escaping points in the dynamic plane of $p_\lambda$.

\begin{theorem}[{\cite[Theorem~2.3]{schleicherzimmer}}]
    If $\lambda \in \Lambda$, then for
every bounded $\underline{s}$ there is a unique injective and continuous curve $\gamma_{\underline{s}}: (0,\infty)\longrightarrow  \C$
of external address $\underline{s}$ satisfying
\begin{align*}
    \lim_{t\rightarrow \infty}
\Re
\gamma_{\underline{s}}(t)
= +\infty
\end{align*}which has the following properties: it consists of escaping points such that
\begin{align*}
    p_{\lambda}(\gamma_{\underline{s}}(t)) &= \gamma_{\underline{s}}(p_\lambda (t)) \hspace{5pt} \forall t>0\\
    \gamma_{\underline{s}}(t)  &= t - c+2\pi i s_1 +r_{\underline{s}}(t)  \hspace{5pt} \forall t>0
\end{align*}
with $|r_{\underline{s}}(t)| < 2e^{
-t}(|K| + C)$, where $C \in \R$
depends only on a bound for $\underline{s}$.

\end{theorem}
The curve $\gamma_{\underline{s}}$ is called the dynamic ray at external address $\underline{s}$.
\begin{theorem}(\cite{combi_classif_exp})
For every preperiodic external address $\underline{s}$ starting with the entry $0$, there
exists a postsingularly finite exponential map $\lambda \exp(z)$ such that the dynamic
ray at external address $\underline{s}$ lands at the singular value $0$.
Every postsingularly finite exponential map  is associated in this way
to a positive finite number of preperiodic external addresses starting
with $0$.
\end{theorem}

As in the polynomial case, we have a notion of parameter rays to the exponential family:
\begin{theorem}[{\cite[Theorem~1.1]{Forster_2007}}]\label{thm:foster rempe schleicher}
    The set of parameters $\lambda  \in \C^*$ for which $p_\lambda ^{\circ n} \rightarrow \infty$ consists of uncountably many disjoint curves in $\C$. More precisely, every path-connected component of this set is an injective curve $\gamma : (0,\infty) \longrightarrow \C$ or $\gamma : [0,\infty) \longrightarrow \C$ with $\lim_{t\rightarrow \infty}\gamma(t) = \infty$.
\end{theorem}
Every path connected component of the set above is referred to as a parameter ray. We say that a parameter ray $\gamma$  lands if $\lim_{t\rightarrow 0} \gamma(t)$ exists. \cite{förster2005parameter} separately constructed parameter rays for the family $\exp(z+\kappa)$, where $\kappa \in \C$. We state a modified version from \cite{Foerster2006}:  
\begin{theorem}(\cite[Theorem~3.9]{Foerster2006})
    For every exponentially bounded sequence $\underline{s} \in \Z^{\N}$, there is a unique curve $\gamma_{\underline{s}} : (t_s,\infty) \longrightarrow \C$ such that every parameter $\kappa  = \gamma_{\underline{s}}(t)$ is an escaping parameter with external address $\underline{s}$ (meaning that in the dynamical plane of $\exp(z+\kappa)$, the ray at address $\underline{s}$ lands at $0$), and potential $t$. Conversely, every parameter $\kappa$ for which the dynamic ray at address $t$ contains $0$ with potential $t>t_s$ satisfies  $\gamma_{\underline{s}}(t) = \kappa$.
\end{theorem}
For the $p_\lambda$ family, we shall extensively use the two following results: 
\begin{theorem}[{\cite[Theorem~3.4]{combi_classif_exp}}]
For every postsingularly finite exponential map $p_\lambda$ and every preperiodic external address $\underline{s}$ with $s_1 = 0$, the dynamic ray at address $\underline{s}$ lands at the singular value if
and only if the parameter ray at address $\underline{s}$ lands at $\lambda$.

\end{theorem}

\begin{theorem}\cite[Corollary~3.5]{combi_classif_exp}
     (Landing Properties of Preperiodic Parameter
Rays)
Every parameter ray  at a strictly preperiodic external address $\underline{s}$ lands at a
postsingularly finite exponential map, and every preperiodic exponential
map is the landing point of a finite positive number of parameter rays
at strictly preperiodic external addresses.

\end{theorem}

If $\gamma_{\underline{s}}$ lands at $\lambda$, then in the dynamic plane of $p_\lambda$, $0$ has external address $\lambda$.

Let $\mathcal{P}$ denote collection of $\lambda \in \C^*$ such that $p_\lambda$ is postsingularly finite. We let $\Theta(\lambda)$ denote the set of external addresses $\underline{s}$ such that the parameter ray $\gamma_{\underline{s}}$ lands at $\lambda$. 

Exponential maps can also be thought of in the form $\exp(z)+c$; dynamic rays for these maps and parameter rays in the $c-$plane are analogously defined.  We note that if $c = \ln \lambda$ and the dynamic ray at address $\underline{s} = s_1s_2.....$ lands at $c$ in the dynamic plane of $\exp(z)+c$, then the dynamic ray at address $0(s_2-s_1)(s_3-s_1)(s_4-s_1)....$ lands at $0$ in the dynamic plane of $p_\lambda$.

\subsection{Degree$-d$ spiders}\label{sec:spiders}
\begin{figure}[t]
\begin{subfigure}[b]{0.45\textwidth}
    \centering
    \includegraphics[width=1.1\textwidth]{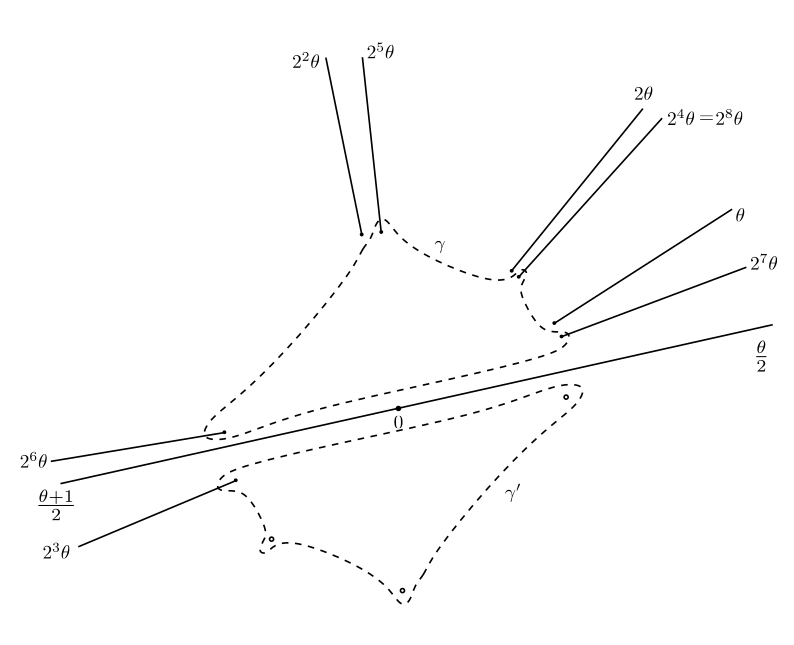}
    \caption{The spider $S_2\Big(\frac{17}{240}\Big)$. The curve $\gamma$ forms a Levy cycle}
\end{subfigure}\hspace{10pt}
\begin{subfigure}[b]{0.45\textwidth}
    
    \includegraphics[width=1.1\textwidth]{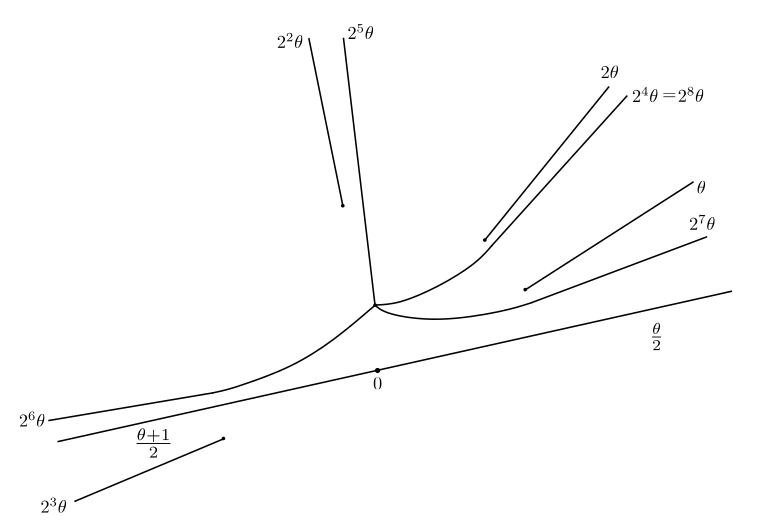}
    \vspace{0.5cm}
    \caption{The quotient spider $\Big[S_2\Big(\frac{17}{240}\Big)\Big]$. The Levy cycle $\gamma$ has been contracted to a point}
\end{subfigure}
\caption{Spiders in degree $2$}
\label{fig:spiders in deg 2}
\end{figure}
Fix $\lambda \in \mathcal{P}_d$. 
In this section, we describe how to construct a topological model for $p_{d,\lambda}$ using angular coordinates for $\lambda$. This construction is done in \cite[Chapter~10]{teichmuller_theory_vol2} and \cite{spideralgo} in the $d=2$ case. \cite{spideralgo} additionally mentions generalizations to higher degrees. 

Given $\theta \in \Theta_d(\lambda)$, we define the spider and extended spider as follows:
\begin{itemize}
    \item For preperiodic $\theta$, the spider $S_d(\theta) = \bigcup_{j\geq 1}\{re^{2\pi i d^{j-1}\theta }: r \in [1,\infty]\}$
    \item For periodic $\theta$ with period $k$, the spider $S_d(\theta)$ is the above set together with the set $\{re^{2 \pi i d^{k-1}\theta} : r \in [0,1]\}$.
    \item The extended spider $\widetilde{S}_d(\theta)$ is the union of the spider and the rays $\{re^{\frac{2\pi i (\theta+j)}{d}}: r \in [0,\infty]\}$ for $j=0,1,...,d-1$
\end{itemize}
A spider is so named because it resembles the eponymous insect perched on $S^2$ with the head placed at $\infty$. Each connected component of $S_d(\theta) \setminus \{\infty\}$ (or  $\widetilde{S}_d(\theta) \setminus \{\infty\}$) is called a \textit{leg}. 
There is a natural map $\mathcal{F}_{d,\theta}:\widetilde{S}_d(\theta)\longrightarrow S_d(\theta)$ that emulates the map $\mu_d$ on $\R/\Z$: for $\theta$ strictly preperiodic under $\mu_d$, given by
\begin{align*}
    \mathcal{F}_{d,\theta}(re^{2\pi it}) & = 
    \begin{cases}
    \infty & r=\infty\\
        e^{2 \pi i\theta} & r=0\\
        re^{2\pi i dt} & dt \neq \theta\\
        (r+1)e^{2 \pi i \theta} & dt = \theta
    \end{cases}
\end{align*}

For $\theta$ periodic under $\mu_d$ with period $k$, the definition of $\mathcal{F}_{d,\theta}$ is the same as above, but we additionally let  $\mathcal{F}_{d,\theta}(re^{2\pi i d^{k-2}\theta}) = (r-1)e^{2\pi i d^{k-1}\theta}$ to make sure that $0$ is periodic. 

The only critical points of $\mathcal{F}_{d,\theta}$ are $0,\infty$, each with local degree $d$. By the Alexander trick (\cite[Chapter~1]{primer}), we can extend $\mathcal{F}_{d,\theta}$ to a $d-$ sheeted branched self-cover of $S^2$. As shown in \cite{spideralgo} and \cite[Chapter~10]{teichmuller_theory_vol2},
\begin{itemize}
    \item If $\theta$ is periodic and an angular coordinate for $\lambda$, $\mathcal{F}_{d,\theta}$ is Thurston equivalent to $p_{d,\lambda}$
    \item If $\theta$ is preperiodic and an angular coordinate for $\lambda$, $\mathcal{F}_{d,\theta}$ is realized if and only if the eventual period of $\theta $ under $\mu_d$ is equal to the eventual period of $\lambda$ under $p_{d,\lambda}$( equivalently the period of the kneading sequence $\Sigma_{d,\theta}(\theta)$).  If equality holds, $\mathcal{F}_{d,\theta}$ is Thurston equivalent to $p_{d,\lambda}$
\end{itemize}
Suppose $\theta$ has preperiod $\ell\geq 1$ and period $k\geq 1$ under $\mu_d$, and $k$ is strictly larger than the eventual period $k'$ of $\lambda$ under $p_{d, \lambda}$, then for each $n\geq \ell$, the points of the form $e^{2\pi i d^n\theta},e^{2\pi i d^{n+k'}\theta},e^{2\pi i d^{n+2k'}\theta},$etc. all have the same itinerary with respect to $\theta$. By drawing loops around all points that share an itinerary, we obtain a Levy cycle for $\mathcal{F}_{d,\theta}$.

However, the topological map obtained by contracting each loop in this Levy 
 cycle to a point is realized. We describe this in detail: form quotient graphs of $S_d(\theta)$ and $\widetilde{S}_d(\theta)$ by identifying the end points $e^{2\pi i n^j\theta}$ with $e^{2\pi i n^{j+k'}\theta}$ for all $j$, without changing the circular order of the legs. Call the new quotient graphs $[S_d(\theta)]$ and $[\widetilde{S}_d(\theta)]$ respectively. Then $\mathcal{F}_{d,\theta}$ descends to map $[\mathcal{F}_{d,\theta}]: [\widetilde{S}_d(\theta)] \longrightarrow [S_d(\theta)]$, and the extension of $[\mathcal{F}_{d,\theta}]$ to $S^2$ is Thurston equivalent to $p_{d,\lambda}$.
\begin{remark}
 This quotient construction  can be done for any $\theta$, so the graph $[S_d(\theta)]$ is defined for all $\theta \in \Q/\Z$. We will assume that we have fixed an embedding of $[S_d(\theta)]$ in $S^2$ for all $d, \theta$.     
\end{remark}
A \textit{leg} of $[S_d(\theta)]$ (or $[\widetilde
{S}_d(\theta)]$) is the image of a leg of $S_d(\theta)$ (resp. $\widetilde{S}_d(\theta)$) under the quotient map. When $d,\theta$ are evident, we use the notation $x_1,x_2,x_3,...$ for the equivalence class  of $e^{2\pi i \theta}$, $\mathcal{F}_{d,\theta}(e^{2\pi i \theta})$,  $\mathcal{F}_{d,\theta}^{\circ 2}(e^{2\pi i \theta})$, etc in $[S_d(\theta)]$.

\begin{definition}
   For $d,d'\geq 2$ and $\theta,\theta' \in \Q/\Z$, we say the spiders $S_d(\theta), S_{d'}(\theta')$ are \textit{isomorphic} if they satisfy the following properties:
    \begin{itemize}
        \item They have the same number of legs
        \item Let $\theta_n = d^{n-1}\theta$ and $\theta'_n = (d')^{n-1}\theta'$. Suppose there are totally $r$ legs and the counterclockwise order of angles in $S_d(\theta)$ is $(\theta_{n_1}, \theta_{n_2},....,\theta_{n_r})$, then the corresponding order in $S_{d'}(\theta')$is  $(\theta'_{n_1}, \theta'_{n_2},....,\theta'_{n_r})$.
    \end{itemize}
\end{definition}
\begin{remark}
    If $S_d(\theta), S_{d'}(\theta')$ are isomorphic, then $[S_d(\theta)], [S_{d'}(\theta')]$ are also isomorphic in the same sense as the above two points.
\end{remark}
\subsubsection{The spider algorithm}
The spider algorithm was first introduced in \cite{spideralgo} for unicritical polynomials. 

Given $\theta \in \Theta_d(\lambda)$, let $A_d(\theta)$ be the set of end points of the legs of $[S_d(\theta)]$ (that is, $A_d(\theta)$ is the image of the set $\{e^{2\pi i d^j\theta}\}_{j\geq 0}$ under the quotient $S_d(\theta) \longrightarrow [S_d(\theta)]$). For $\lambda \in \mathcal{P}_d$, we note that $|A_d(\theta)| \geq 2$.  We define the spider space $\mathcal{T}^d_{\theta}$ as follows:
\begin{align*}
    \mathcal{T}^d_{\theta}& =\{\phi: [S_d(\theta)] \hookrightarrow \hat{\C}  |  \phi(\infty) = \infty,\\ &\hspace{17 pt}\phi \text{ is orientation preserving
     and preserves circular order of legs at }\infty\}/\sim
\end{align*}
where $\phi \sim \psi$ if, for some affine map $M$, $\phi$ is isotopic to $M \circ \psi$ rel $A_d(\theta)$. In \cite{spideralgo}, the authors define spider space as the space of maps $\phi: S_d(\theta)\longrightarrow \hat{\C}$ under the same equivalence instead. In the case of periodic $\theta$, for instance, $S_d(\theta)$ and $[S_d(\theta)]$ are the same graphs, and the two definitions coincide. 
\begin{proposition}\label{prop:isom teich spaces}
    $\mathcal{T}^d_\theta$ is isomorphic to the Teichm\"{u}ller space $\Teich(S^2,A_d(\theta)\cup \{\infty\})$, 

\end{proposition}
\begin{proof}
    Define a map $H:\Teich(S^2,A_d(\theta)\cup \{\infty\}) \longrightarrow  \mathcal{T}^d_\theta$ by $H([\phi]) = [\phi|[S_d(\theta)]]$ for all $[\phi] \in \Teich(S^2,A_d(\theta)\cup \{\infty\})$. 

If $\psi$ is isotopic to $M\circ \phi$ relative to $A_d(\theta)$, then $\psi|[S_d(\theta)]$ is isotopic to $M \circ \phi | [S_d(\theta)] $ relative to $A_d(\theta)$, therefore this map is well-defined. Conversely, given $[\phi] \in \mathcal{T}^d_\theta$, we can extend $\phi$ to $S^2$ using the Alexander trick to get an element in $\Teich(S^2,A_d(\theta)\cup \{\infty\})$. It is easy to see that this map is also well-defined, and forms an inverse for $H$.
\end{proof}
Thus the Teichm\"{u}ller distance on $\mathcal{T}^d_\theta$ is given as 
\begin{align*}
      d_{T}([\phi], [\psi]) & = \inf_{q.c. \psi' \sim \phi \circ \psi^{-1}}\ln K(\psi')
\end{align*}
for any $[\phi], [\psi] \in \mathcal{T}^d_\theta$, where $\psi' \sim \phi \circ \psi^{-1}$ means $\psi'$ is isotopic to $\phi \circ \psi^{-1}$ rel $A_d(\theta) \cup \{\infty\}$, and $K(\psi')$ is the complex dilitation of the quasiconformal map $\psi$. 

  On $\mathcal{T}^d_{\theta}$, the spider operator $\sigma^d_\theta$ is defined as follows: given $[\phi] \in \mathcal{T}^d_\theta$, with the representative $\phi$ chosen so that $\phi(x_1) = 0$, let $\sigma^d_\theta([\phi]) = [\psi]$, where $\psi$ completes the following diagram:
\[
\begin{tikzcd}
\Big([S_d(\theta)],A_d(\theta)\Big )
\arrow[r,"\psi"] \arrow[d,"{[\mathcal{F}_{d,\theta}]}"]&\Big(\hat{\C},\psi(A_d(\theta))\Big) \arrow[d,"p_\phi"]\\
\Big ([S_d(\theta)],A_d(\theta)\Big ) \arrow[r,"\phi"] & \Big (\hat{\C},\phi(A_d(\theta))\Big )
\end{tikzcd}
\]
where $p_\phi(z)=\phi(x_2)\big(1+\frac{z}{d}\big)^d$.

The operator $\sigma^d_\theta$ is well-defined, and is equal to $H \circ \sigma_{[\mathcal{F}_{d,\theta}]} \circ H^{-1}$. It is thus weakly contracting with respect to the Teichm\"{u}ller metric on $\mathcal{T}^d_\theta$. An alternate proof that does not use the relationship between $\sigma^d_\theta$ and $\sigma_{[\mathcal{F}_{d,\theta}]}$ is similar to the proof of the weak contracting property for the spider operator defined in \cite{spideralgo}. The map $\sigma^d_\theta$ has a fixed point $[\phi] \in \mathcal{T}^d_\theta$, where $\phi$ is chosen so that  $\phi(e^{2\pi i \theta}) = 0$, and there exists $\psi \in [\phi]$ satisfying $\phi \circ [\mathcal{F}_{d,\theta}]\circ \psi^{-1} = p_{d,\lambda}$. 
\subsection{Combinatorial embeddings}\label{sec: combi embeddings defn}
We introduce a partial order on $\mathcal{P}_d$ as follows: $\lambda < \mu$ if there exist angles $\theta,\theta' \in \Theta_d(\lambda)$ that land at the same point in $\mathcal{M}_d$ and $\alpha \in \Theta_d(\mu)$ such that $\theta < \alpha < \theta'$.

\begin{remark}
    Note that this definition is equivalent to the following. $\lambda < \mu$ if and only if there exist $c \in M_d(\lambda)$, $c' \in M_d(\mu)$ such that 
    \begin{itemize}
        \item $\lambda$ is critically periodic, and $c'$ is in the wake of $c$
        \item $\lambda$ is Misiurewicz (ie, critically preperiodic), and there exist angles $\theta,\theta' \in \Omega_d(c)$ and $\alpha \in \Omega_d(c')$ such that $\theta < \alpha < \theta'$. 
    \end{itemize}
\end{remark}
\begin{figure}[t]
 \begin{subfigure}[b]{0.4\textwidth}
        \centering
    \includegraphics[scale=0.3]{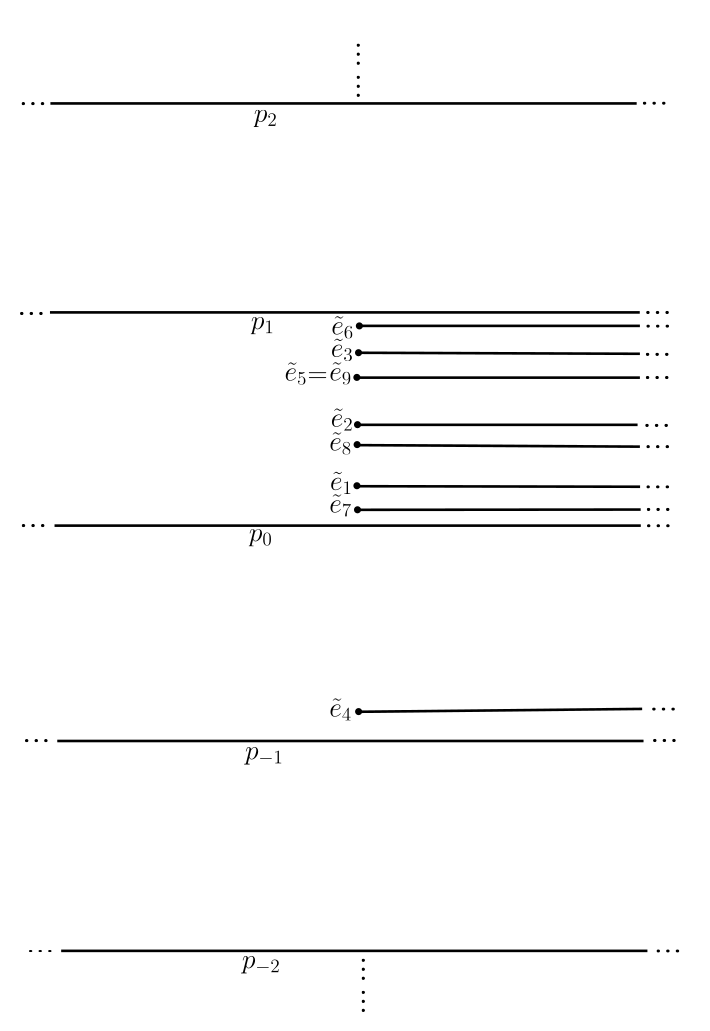}
 \end{subfigure}
  \begin{subfigure}[b]{0.4\textwidth}
        \centering
    \includegraphics[scale=0.3]{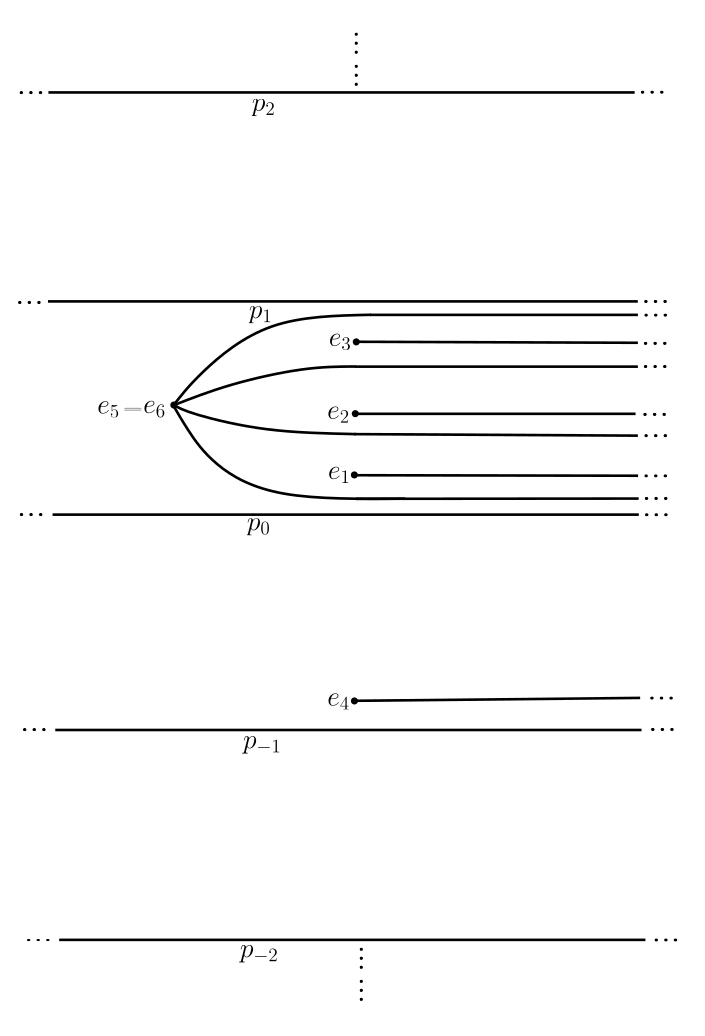}
 \end{subfigure}
    \caption{The spiders $\Gamma_{\underline{s}}$ and $[\Gamma_{\underline{s}}]$ for $s = 0001\overline{0010}$}
    \label{fig:exp_spider}
\end{figure}
\begin{definition}\label{defn:combi_embedding}

    Given an injective map  $\Phi: \mathcal{P}_d \hookrightarrow \mathcal{P}_{d+1}$, we call $\Phi$ a \textit{combinatorial embedding} if it satisfies the following properties:
    \begin{enumerate}
    \item There exists a bijection $h: \Post(p_{d,\lambda}) \longrightarrow \Post(p_{d+1,\Phi(\lambda)})$ such that $h(0) = 0$ and  $ h \circ p_{d,\lambda} = p_{d+1,\Phi(\lambda)} \circ h$
        \item For $\lambda \in \mathcal{P}_d$, and $\theta \in \Theta_d(\lambda)$, there exists an angle $\phi \in \Theta_{d+1}(\Phi(\lambda))$ such that $S_d(\theta)$ and $S_{d+1}(\phi)$ are isomorphic
        \item If $\lambda$ is in a satellite component of a component containing $\mu$, then $\Phi(\lambda)$ is in a satellite component of the component containing $\Phi(\mu)$
        \item If $\lambda < \mu$, then $\Phi(\lambda) < \Phi(\mu)$
    \end{enumerate}
\end{definition}

\subsection{Spiders for exponential maps}
\subsubsection{Spiders corresponding to preperiodic addresses}\label{sec:exp spider defn}
Given a strictly preperiodic external address $\underline{s} $, we recall the construction of the spider graph $\Gamma_{\underline{s}}$ in 
\cite[Section~5.1]{combi_classif_exp} here.  Let $\ell$ and $k$ be the
preperiod and period of $\underline{s}$ respectively.

Start with a vertex $\tilde{e}_{\infty}$, assumed to be at $+\infty$, and the vertex $\tilde{e}_1 = 0$. For each $n \in \{2,...,\ell+k\}$, add a vertex $\tilde{e}_n$ lying on the $y-$axis and belonging to the sector $\{(2s_n-1)\pi < \Im z < (2s_n+1)\pi\}$ such that  $\Im \tilde{e}_n< \Im \tilde{e}_m$ if  $\sigma^{\circ n}(\underline{s}) < \sigma^{\circ m}(\underline{s})$, and $\tilde{e}_n = \tilde{e}_m$ if $\sigma^{\circ n} (\underline{s}) = \sigma^{\circ m} (\underline{s}) $. For each $n$ let  $\gamma_n$ be the horizontal ray $\{x\geq 0, y = \Im \tilde{e}_n\}$ from $\tilde{e}_n$ to $\tilde{e}_\infty$. This completes the definition of $\Gamma_{\underline{s}}$.

Define the augmented spider $\widetilde{\Gamma}_{\underline{s}}$ as follows. First add a vertex called $\tilde{e}_{-\infty}$, taken to be at $-\infty$. Then for each $n \in \Z$, add the edge $p_n = \{y= (2n-1)\pi\}$ joining $\tilde{e}_{-\infty}$ and $\tilde{e}_\infty$. 

Define a graph map $\mathcal{F}_{\underline{s}}: \widetilde{\Gamma}_{\underline{s}}\longrightarrow \Gamma_{\underline{s}}$ by mapping  $\tilde{e}_{\infty}$ to $\tilde{e}_{\infty}$, $\tilde{e}_{-\infty}$ to $\tilde{e}_1 = 0$, and $\tilde{e}_n$ to $\tilde{e}_{n+1}$ for all $n$. Each edge $p_n$ is mapped onto $\gamma_1$, and every $\gamma_n$ is mapped onto $\gamma_{n+1}$. We can use the Alexander trick to extend $\mathcal{F}_{\underline{s}}$ to a map $\mathcal{F}_{\underline{s}} : S^2 \setminus \{\infty\} \longrightarrow
S^2 \setminus \{ \infty\}$ whose postsingular set is $\{\tilde{e}_1,....,\tilde{e}_{\ell+k}\}$. 

Suppose the parameter ray $\gamma_{\underline{s}}$ lands at $\lambda \in \C^*$:
\begin{itemize}
    \item If the eventual period of the orbit of $\lambda$ (equivalently the eventual period of $\Sigma_{\underline{s}}(\underline{s})$) is $k$ , then $\mathcal{F}_{\underline{s}}$ is Thurston equivalent to $p_{\lambda}$
    \item If the eventual period of the orbit of $\lambda$ is $k' < k$, then $\mathcal{F}_{\underline{s}}$ has a Levy cycle; however, by collapsing this Levy cycle we may form a map equivalent to $p_\lambda$. More precisely, form the quotient graphs $[\widetilde{\Gamma}_{\underline{s}}]$ and $[\Gamma_{\underline{s}}]$ by identifying the vertices $\tilde{e}_n$ and $\tilde{e}_{n+k'}$ for $n=\ell, \ell+1,\ell+2,etc$ without changing the vertical order of edges. Then $\mathcal{F}_{\underline{s}}$ descends to a map $[\mathcal{F}_{\underline{s}}]: [\widetilde{\Gamma}_{\underline{s}}] \longrightarrow [\Gamma_{\underline{s}}]$, which can again be extended to $S^2 \setminus \{\infty\}$. This extended map $[\mathcal{F}_{\underline{s}}]$ is a topological exponential map that is Thurston equivalent to $p_\lambda$. 
\end{itemize}
Note that the quotient construction in the second point can be done for any $\underline{s}$, so the graph $[\Gamma_{\underline{s}}]$ is defined for all $\underline{s}$. We define the legs of $[\Gamma_{\underline{s}}]$ to be the images of the edges of $\Gamma_{\underline{s}}$. When $\underline{s}$ is evident, we let $e_n$ denote the class of $\tilde{e}_n$ in $[\Gamma_{\underline{s}}]$, and let $A_{\underline{s}} = \{e_1,...,e_{\ell+k}\}$. We note that $|A_{\underline{s}}| \geq 2$ for all adresses $\underline{s}$.

Draw a line  $\beta = \{ y=b\}$ disjoint from all the legs of $\widetilde{\Gamma}_{\underline{s}}$,  where $b$ is chosen so that  $b < \Im \tilde{e}_m$ if and only if $\overline{0} < \sigma^{\circ (m-1)}(\underline{s})$, and $\beta $ is bound between $p_r$ and $p_{r+1}$ if $r\overline{0} < \beta < (r+1)\overline{0}$. The edge $\beta$ is supposed to represent the lexicographic position of the external address $\overline{0}$. 
We note that $\sigma^{\circ (n-1)}(\underline{s}) << \sigma^{\circ (m-1)}(\underline{s})$ if and only if, starting from $\gamma_{n}$ and moving counterclockwise in a neighborhood of $\infty$, we can reach $\gamma_{m}$  without intersecting $\beta$. This gives a circular order of the legs of $\Gamma_{\underline{s}}$ (and $[\Gamma_{\underline{s}}]$).

\subsubsection{A spider algorithm for exponentials}\label{sec:spider algo for exp}
Given $\lambda \in \mathcal{P}$ with $\underline{s} \in \Theta(\lambda)$, we define the spider space $\mathcal{T}_{\underline{s}}$ as follows:
\begin{align*}
\mathcal{T}_{\underline{s}}& =\{\phi: [\Gamma_{\underline{s}}] \hookrightarrow \hat{\C}  |  \phi(\infty) = \infty\\ &\hspace{17 pt}\phi \text{ is orientation preserving
     and preserves circular order of legs at }\infty\}/\sim
\end{align*}
where $\phi \sim \psi$ if, for some affine map $M$, $\phi$ is isotopic to $M \circ \psi$ rel $A_{\underline{s}}$. 
\begin{proposition}
$\mathcal{T}_{\underline{s}}$ is isomorphic to the Teichm\"{u}ller space $\Teich(S^2,A_{\underline{s}}\cup \{\infty\})$.    
\end{proposition}
\begin{proof}
The proof is the same as that of Proposition~\ref{prop:isom teich spaces}. 
\end{proof}
We have an analogous definition of the Teichm\"{u}ller metric on $\mathcal{T}_{\underline{s}}$.

  Given $[\phi] \in \mathcal{T}_{\underline{s}}$ chosen such that $\phi(e_1) = 0$, the spider operator $\sigma_{\underline{s}}$ sends $[\phi]$ to $[\psi] \in \mathcal{T}_{\underline{s}}$, where $\psi$ completes the following diagram:
\[
\begin{tikzcd}
\Big([\Gamma_{\underline{s}}],A_{\underline{s}}\Big) \arrow[r,"\psi"] \arrow[d,"{[\mathcal{F}_{\underline{s}}]}"]&\Big(\C,\psi(A_{\underline{s}})\Big) \arrow[d,"p_\phi"]\\
\Big([S_{\underline{s}}],A_{\underline{s}}\Big) \arrow[r,"\phi"] & \Big(\C,\phi(A_{\underline{s}})\Big)
\end{tikzcd}
\]
where $p_\phi(z) = \phi(e_2)\exp(z)$.
\begin{proposition}\label{prop: strong contraction}
    The operator $\sigma_{\underline{s}}$ is weakly contracting on $\mathcal{T}_{\underline{s}}$ and locally uniformly contracting.  
\end{proposition}
\begin{proof}
   The operator $\sigma_{\underline{s}}$ is conjugate to the Thurston pullback operator $\sigma_{[\mathcal{F}_{\underline{s}}]}$, which we know from Section~\ref{sec:thurston pullback} is weakly contracting on $\mathcal{T}_{\underline{s}}$ and locally uniformly contracting.
\end{proof}
The operator $\sigma_{\underline{s}}$ 
 has a fixed point $[\phi]$, where $\phi(e_1) = 0$ and there exists $\psi\in [\phi]$ such that $\phi \circ [\mathcal{F}_{\underline{s}}] \circ \psi^{-1}(z) = \lambda \exp(z)$ (for an alternate proof, see \cite{combi_classif_exp}). 
\begin{definition}
    Given a finite set $A \subset S^2$, the pure mapping class group $\PMCG(S^2,A)$ is the group of orientation preserving homeomorphisms $\phi: (S^2,A) \longrightarrow (S^2,A)$ where $\phi|A = \textrm{id}_A$, with $\phi$ and $\psi$ equivalent if $\phi$ is isotopic to $\psi$ relative to $A$. The group law is given by function composition. We let $\langle h\rangle $ denote the equivalence class of $h$ in $\PMCG(S^2,A)$. 
\end{definition}
$\PMCG(S^2,A)$ inherits the topology of uniform convergence on compact subsets from $\text{Homeo}^+(S^2,A)$. It is known that in this topology, $\PMCG(S^2,A)$ is a discrete group (see for example \cite{primer}).  

\begin{definition}
The moduli space $\Mod(\underline{s}) = \{\phi: A_{\underline{s}} \longrightarrow \C\}/~$ where $\phi \sim \psi$ if $\phi = M \circ \psi$ for some affine map $M$. 
\end{definition}
We let $[[\phi]]$ denote the equivalence class of $\phi$ in $\Mod(\underline{s})$.  It is known that $\Mod(\underline{s})$ is isomorphic to an open subset of $\C^{|A_{\underline{s}}|-2}$ and that the map $\pi: \mathcal{T}_{\underline{s}} \longrightarrow \Mod(\underline{s})$ given by  $\pi([\phi])=[[\phi|A_{\underline{s}} \cup \{\infty\}]]$ is a universal cover (the proof is similar to that of \cite[Proposition~3.1]{spideralgo}). Suppose $A_{\underline{s}} = \{e_1,...,e_r\}$. If we fix a representative $\phi$ such that $\phi(e_1) = 0$, then $\pi([\phi]) = [[\phi]] = \Big(\frac{\phi(e_3)}{\phi(e_2)}, \frac{\phi(e_4)}{\phi(e_2)},...,\frac{\phi(e_r)}{\phi(e_2)}\Big)$.

\begin{proposition}\label{prop:fiber over moduli space}
Given $[\phi] \in \mathcal{T}_{\underline{s}}$, the fiber $\pi^{-1}([[\phi|A_{\underline{s}} \cup \{\infty\}]])$ is the collection of elements of the form $[h \circ \phi]$, where $\langle h\rangle  \in \PMCG(S^2,\phi(A_{\underline{s}}\cup \{\infty\}))$.  
\end{proposition}
\begin{proof}

    Without loss of generality, assume that $\phi(e_1) = 0$ and $\phi(e_2) = 1$. Then $[[\phi|A_{\underline{s}} \cup \{\infty\}]] = (\phi(e_3),....,\phi(e_r))$. It is clear that $\pi([h \circ \phi]) = [\phi]$ for all $\langle h\rangle  \in \PMCG(S^2,\phi(A_{\underline{s}}\cup \{\infty\}))$.
    
 Suppose $[\psi] \in \pi^{-1}([[\phi|A_{\underline{s}} \cup \{\infty\}]])$, if we assume without loss of generality that $\psi(e_1) = 0$ and $\psi(e_2) = 1$, then $\psi | A = \phi|A$, in other words, $\langle \psi \circ \phi^{-1}\rangle \in \PMCG(S^2,A_{\underline{s}}\cup \{\infty\})$. 
\end{proof}

\section{Proof of Lemma~\ref{lemma:injective_maps_of_angles}}\label{sec:proofoflemma1.1}
Throughout this section we fix a degree $d\geq 2$.
\subsection{Construction of $Z_j$}
For any angle $\theta$, we let $\theta = d^{n-1}\theta$. If $\theta$ is $k-$periodic under $\mu_d$, $\theta = \frac{\theta+j_\theta}{d}$ for a unique $j_\theta \in \{0,1,...,d-1\}$. If $\theta$ is strictly preperiodic, then $\mathscr{O}_d(\theta)$ contains no element in $\mu_d^{-1}(\theta)$.

For a given $j \in \{0,1,...,d-1\}$, define a `symbol shift' function \begin{align*}
    u_{j,\theta}:\R/\Z \setminus \Big\{\frac{\theta+j}{d}\Big\}& \longrightarrow \{0,1,...,d\}\\
    u_{j,\theta}(t) & = 
    \begin{cases}
    m & t \in \big[\frac{m}{d},\frac{m+1}{d}\big), m<j\\
    m+1 & t \in \big[\frac{m}{d},\frac{m+1}{d}\big), m>j\\
    j & t \in \big[\frac{j}{d},\frac{\theta+j}{d}\big)\\
    j+1 &t \in \big(\frac{\theta+j}{d},\frac{j+1}{d}\big)
    \end{cases}
    \end{align*}
    Additionally, if 
    $\frac{\theta+j}{d} \in \mathscr{O}_d(\theta)$, let \begin{align*}
    u_{j,\theta}\Biggl(\frac{\theta+j}{d}\Biggl)&=\begin{cases}j & \text{ if }\theta \text{ is the smaller angle in a companion pair}\\ j+1 &\text{otherwise}
    \end{cases}
    \end{align*}
 The $u_{j,\theta}$ function assigns  a symbol to each angle in $\mathscr{O}_d(\theta)$: we first divide $[0,1)$ into $d$ sub-intervals of the form $\big[\frac{m}{d},\frac{m+1}{d}\big)$. The symbol assigned to an angle $t$ depends on which sub-interval $t$ belongs to. The sub-interval $\big[\frac{j}{d},\frac{j+1}{d}\big)$ is `split' at $\frac{\theta+j}{d}$, and the symbols to the left and right of this angle differ by $1$. 

The goal is to push angles in $\big(\frac{\theta+j}{d},1\big)$ a sub-interval further. Define $Z_j \colon \Q/\Z  \righttoleftarrow$ by 
\begin{align*}
    Z_j(\theta) & = \sum_{n=1}^\infty\frac{u_{j,\theta}(\theta_n)}{(d+1)^n} 
\end{align*} The angle $Z_j(\theta)$ has $(d+1)$-adic expansion  $.\overline{u_{j,\theta}(\theta_1)u_{j,\theta}(\theta_2)u_{j,\theta}(\theta_3)...}$.
\begin{remark}\label{noendd}
Suppose $u_{j,\theta}(\theta_n) = d$ for $n\geq N$, then $\theta_{n} \in [\frac{d-1}{d},1)$ $\forall n\geq N$. This forces $\theta_N \equiv 0$, and thus, $u_{j,\theta}(\theta_N) = 0$, contradicting our assumption. Thus the $d+1-$adic expansion of $Z_j(\theta)$ produced by the symbol shift function does not end in a constant spring of $d$'s.  \end{remark}
\begin{example}
    Let $d=2$. The angle $\theta = \frac{17}{2^4(2^4-1)} = \frac{17}{240}$ is strictly preperiodic under $\mu_2$, with preperiod $4$ and period $4$. The sequence $.0001\overline{0010}$ is a $2-$adic expansion for $\theta$. 

    Applying the definition above, we get $3-$adic expansions for $Z_0(\theta)$ and $Z_1(\theta)$ as below:
    \begin{align*}
        Z_0(\theta) & = .1112\overline{1121}  = \frac{3323}{3^4(3^4-1)} = \frac{3323}{6480}\\
        Z_1(\theta)& = .0002\overline{0010}  = \frac{163}{3^4(3^4-1)} = \frac{163}{6480}
    \end{align*}
\end{example}
\begin{example}
    Let $d=2$. For $\theta = \frac{1}{7} = .\overline{001}$, we have 
    \begin{align*}
        Z_0(\theta)& = .\overline{112} = \frac{14}{3^3} = \frac{14}{26}\\
        Z_1(\theta)& = .\overline{001} = \frac{1}{3^3} = \frac{1}{26}
    \end{align*}
\end{example}
\subsection{Monotonicity of $Z_j$}
Fix $j \in \{0,1,...,d-1\}$. In this section we show that $Z_j$ is injective and order-preserving. 
\begin{proposition}\label{prop:monotonedyn}\label{domainvar}\label{sectorspreserved}
Given $s,t \in \R/\Z$ such that $u_{j,\theta}(s_n)$ and $u_{j,\theta}(t_n)$ are defined for all $n$, let
\begin{align*}
    s' & = .u_{j,\theta}(s_1)u_{j,\theta}(s_2)u_{j,\theta}(s_3)....\\
    t' & = .u_{j,\theta}(t_1)u_{j,\theta}(t_2)u_{j,\theta}(t_3)....
\end{align*}
where the right hand side is interpreted as a $d+1$ - adic expansion. Then $s<t$ if and only if $s'<t'$.
\end{proposition}
\begin{proof}
Choose $d-$adic expansions $s=.x_1x_2x_3....$ and $t = .y_1y_2y_3.....$ so that if either angle is rational, the corresponding expansion does not end in a constant sequence of $(d-1)$'s.

At the first position $r$ where $x_r \neq y_r$, we must have $x_r< y_r$, and so, $u_{j,\theta}(s_r)$ is less than $u_{j,\theta}(t_r)$. For all indices $n$ smaller than $r$, $x_n $ is less than or equal to $ y_n$, and thus,  $u_{j,\theta}(s_n)$ is less than equal to  $u_{j,\theta}(t_n)$.  This shows that $s'$ is less than or equal to $t'$. These two angles are equal if and only if the following conditions are satisfied:
\begin{itemize}
    \item $u_{j,\theta}(s_r) +1 = u_{j,\theta}(t_r)$
    \item For all $n>r$, $u_{j,\theta}(s_n) = d$ and $u_{j,\theta}(t_n ) = 0$
\end{itemize}
By Remark~\ref{noendd}, the second condition is never true. Thus, $s'$ is strictly less than $t'$.
\end{proof}
\begin{proposition}\label{prop:symmetrypres}
Given $\theta \in \Q/\Z$, $Z_j\big(\theta+\frac{1}{d-1}\big) = Z_j(\theta)+\frac{1}{d}$
\end{proposition}
\begin{proof}
Let $\theta' = \theta + \frac{1}{d-1}$. 
For all $n$, $u_{j,\theta'}(\theta_n') = u_{j,\theta}(\theta_n)+1$. 
\end{proof}
\begin{proposition}\label{prop:monotone}
$Z_j$ is strictly increasing. 
\end{proposition}
\begin{proof}
We show this in three steps.
\begin{itemize}
    \item $Z_j(\frac{m}{d-1})$ is equal to $\frac{m}{d}$ if $m<j$, and $\frac{m+1}{d}$ otherwise. This is clear by the fact that $Z_j(\theta) = .\overline{m}$ if $m<j$ and $.\overline{m+1}$ otherwise.\\
\item 
Given rational angles $\omega<\omega'$ in $[0,\frac{1}{d-1})$, by an argument similar to the one in Proposition~\ref{prop:monotonedyn}, $Z_j(\omega)$ is less than $ Z_j(\omega')$. By Proposition~\ref{prop:symmetrypres}, $Z_j \big |\Q \cap {\big[\frac{m}{d-1},\frac{m+1}{d-1}\big)}$ preserves linear order for $0\leq m \leq d-2$.\\
\item If $\theta$ is less than $\frac{1}{d-1}$, at the first position $r$ where $x_r \neq 1$, we must have $x_r = 0$. For all  $n<r$, $u_{j,\theta}(\theta_n) = u_{j,\frac{1}{d-1}}(\frac{1}{d-1})$.

If $j>0$, then $u_{j,\theta}(\theta_n) = 0$ whereas $u_{j,\frac{1}{d-1}}(\frac{1}{d-1})>0$. If $j=0$, then $u_{j,\theta}(\theta_n)=1$ whereas $u_{j,\frac{1}{d-1}}(\frac{1}{d-1})=2$. This shows that $Z_j(\theta)$ is less than $Z_j(\frac{1}{d-1})$.

\end{itemize}
Combining these points yields the statement of the proposition.
\end{proof}
\subsection{Preserving spiders}\label{sec:preserving spiders}
Next, we show that for any rational angle $\theta$ and any $j$, the spiders of $\theta$ and $Z_j(\theta)$ are isomorphic. For the rest of this section, we let $\phi = Z_j(\theta)$, and let $.{x_1x_2x_3....}$ be a $d-$adic expansion for $\theta$ that does not end in a constant string of $(d-1)$'s. 
\begin{proposition}\label{prop:missedsector}
\begin{enumerate}
    \item The angles in $\mathscr{O}_d(\theta)$ and $\mathscr{O}_{d+1}(\phi)$ are in the same circular order. In particular, the preperiod and period of $\phi$ under $\mu_{d+1}$ coincide with the preperiod and period of $\theta$ under $\mu_d$.  
    \item The  orbit of $\phi$ does not intersect the interior of $T^{stat}_{d+1,j}(\phi)$. 
\end{enumerate}
\end{proposition}
\begin{proof}
\begin{enumerate}
    \item follows directly from Proposition~\ref{prop:monotonedyn}. 
    \item  If $\theta_n$ is in $\big(\frac{m}{d},\frac{m+1}{d}\big)$, then $\phi_n$ is in $\big(\frac{m}{d+1},\frac{m+1}{d+1}\big)$ if $m<j$, and in $\big(\frac{m+1}{d+1},\frac{m+2}{d+1}\big)$  if $m>j$. 
    
    If $\theta_n$ is in $\big(\frac{m}{d},\frac{\theta+m}{d}\big)$ for $m<j$, then $\theta_{n+1}$ is less than $\theta$ and so, $\phi_{n+1}$ is less than $\phi$. But this implies that $\phi_n$ is in $\big(\frac{m}{d+1},\frac{\phi+m}{d+1}\big)$. 
    
    Using similar arguments,
          \begin{itemize}
              \item the sector $T^{stat}_{d,m}(\theta)$ corresponds to $T^{stat}_{d+1,m}(\phi)$ for $m<j$, and to $T^{stat}_{d+1,m+1}(\phi)$ for  $m>j$, and the orbit points in these sectors are in the same circular order,
              \item the sector $T^{stat}_{d,j}(\theta)$
              corresponds to  $T^{stat}_{d+1,j+1}(\phi)$, and circular order within these sectors is preserved. If there is an orbit point in $\mathscr{O}_d(\theta)$ on the boundary (which would have to be $\theta_k$), if $\theta$ is the smaller angle in a companion pair,  then $\phi_k = \frac{\phi+j}{d+1} \in \partial T^{stat}_{d+1,j-1}(\phi)\cap \partial T^{stat}_{d+1,j}(\phi)$.  Else, $\phi_k = \frac{\phi+j+1}{d+1} \in \partial T^{stat}_{d+1,j}(\phi)\cap \partial T^{stat}_{d+1,j+1}(\phi)$. 
          \end{itemize}
\end{enumerate}
\end{proof}
\begin{remark}
    Proposition~\ref{prop:missedsector} implies that $S_d(\theta)$ and $S_{d+1}(\varphi)$ are isomorphic.

\end{remark}
\begin{proposition}\label{prop:spiderconjugacy}
  $\mathcal{F}_{d,\theta}|S_{d}(\theta)$ and $\mathcal{F}_{d+1,\varphi}|S_{d+1}(\varphi)$ are conjugate are conjugate by a map $h$ that satisfies 
  \begin{enumerate}
      \item $h(\infty) = \infty$
      \item $h(e^{2\pi i \theta}) = e^{2\pi i \phi}$
      \item $h$ preserves the circular order of legs
  \end{enumerate}
\end{proposition}
\begin{proof}
Let $\theta_n = d^{n-1}\theta$ and $\phi_n = (d+1)^{n-1}\phi$. 
  Let $h(re^{2\pi i \phi_n}) = re^{2\pi i \phi_n}$. By Proposition~\ref{prop:missedsector}implies that $h$ satisfies the required properties.
\end{proof}
By construction of the quotient graphs $[S_d(\theta)]$ and $[S_{d+1}(\phi)]$, the following proposition is clear.
\begin{proposition}\label{prop:spidermapsareconjugate}
   $[\mathcal{F}_{d,\theta}]\Big|[S_d(\theta)]$ and $[\mathcal{F}_{d+1,\phi}]\Big|[S_{d+1}(\phi)]$ are conjugate by a map $h$ such that $h(\infty) = \infty$, $h$ preserves the circular order of legs and  $h(x_1) = \overline{x}_1$,  where $x_1$ is the class of $e^{2\pi i \theta}$ in $[S_d(\theta)]$, and $\overline{x}_1$ is the conjugacy class of $e^{2\pi i \phi}$ in $[S_d(\theta')]$.
\end{proposition}
\begin{remark}
    In later sections, we will simply say that $[\mathcal{F}_{d,\theta}]\Big|[S_d(\theta)]$ and $[\mathcal{F}_{d+1,\phi}]\Big|[S_{d+1}(\phi)]$ are conjugate assume that the conjugating map satisfy the three properties above. 
\end{remark}
\begin{definition}
A degree $d$ \textit{generalised spider} is a pair $(S,f)$ where $S \subset S^2$ is of the form $\gamma_1 \cup \gamma_2 \cup ....\cup \gamma_r \cup \eta_1 \cup \eta_2 \cup..\cup \eta_d$, with $d,r\geq 2$, such that
\begin{enumerate}
    \item $\gamma_j$s are disjoint simple arcs from $\infty$ to the unit disk
    \item $\eta_j$'s are disjoint simple arcs from $\infty$ to $0$
    \item $f: S \longrightarrow S$ is a continuous map critical points $0,\infty$, such that 
    \begin{itemize}
        \item $f(\gamma_j) = \gamma_{j+1 \text{ mod } d}$
        \item $f(\eta_j) = \gamma_1$ for all $j$
        \item $f(\infty) = \infty$, and the $f-$orbit of $0$ is finite
        \item $f$ preserves circular order of all arcs at $\infty$
        \item $S^2 \setminus f(S)$ is a topological disk
    \end{itemize}
\end{enumerate}
\end{definition}
The dynamical system  $(S,f)$ can be extended to $(S^2,f)$ by the Alexander trick so that $f$ is a postcritically finite degree $d$ branched self-covering of $S^2$ with critical points $0,\infty$.
\begin{proposition}\label{prop:genspider}Given two degree $d$ generalised spiders $(S,f), (S',g)$, if the spiders $S,S'$ are homeomorphic via $h$ such that
\begin{enumerate}
    \item $h(0) = 0$
    \item $h(\infty) = \infty$
    \item $h\big|{f(S)} $ is a homeomorphism of $f(S)$ onto $g(S')$
    \item $g\circ h = h \circ f$ upto isotopy,
\end{enumerate}
Then the extended Thurston maps $f,g$ on $S^2$ are Thurston equivalent.
\end{proposition}
\begin{proof}
$h$ can be extended to a homeomorphism from $S^2$ to $S^2$, so that each connected component of $S^2 \setminus S$ maps under $h$ homeomorphically onto a connected component of $S^2 \setminus S'$. Let $D$ be an open connected component of $S^2 \setminus S'$. On $D$, let $ \widetilde{g} = h^{-1} \circ f \circ h$. Since $\widetilde{g}$ agrees with $g$ on $\partial D$,  it is isotopic to $g$ relative to $\partial D$. We can similarly define $\widetilde{g}$ on every connected component of $S^2 \setminus S'$ to get $\widetilde{g}: S^2 \setminus S^2$ such that $\widetilde{g} \circ h = h \circ f$, with $\widetilde{g}$ isotopic to $g$ relative to $P_g$.
\end{proof}

We give an alternate topological description of the map $\mathcal{F}_{d+1,\varphi}$ extended to $S^2$ by the Alexander trick.

Choose $\epsilon >0$ small enough so that 
the interval $[\frac{\theta+j}{d}-\epsilon, \frac{\theta+j}{d}+\epsilon]$ does not contain any element of $\mathscr{O}_{d}(\theta)$ if $j \neq j_\theta$, and only contains $\theta_k$ if $j = j_\theta$.
\begin{itemize}
    \item[-]Suppose $\theta$ is strictly preperiodic under $\mu_d$, or $j \neq j_\theta$, we change $\widetilde{S}_d(\theta)$ by adding a new ray at the angle  $\frac{\theta+j+\epsilon}{d}$. That is, we construct a `tweaked' spider $\widetilde{S}^{j,\epsilon}_d(\theta)$ as 
        \begin{align*}
            \widetilde{S}^{j,\epsilon}_d(\theta): = \widetilde{S}_d(\theta)  \cup   \Big\{re^{2\pi i\omega}: \omega = \frac{\theta+j+\epsilon}{d}, 0\leq r \leq \infty\Big\}
        \end{align*}
Define $\mathcal{G}_{d,j,\theta}: \widetilde{S}^{j,\epsilon}_d(\theta) \longrightarrow S_d(\theta)$ as
        \begin{align*}
            \mathcal{G}_{d,j,\theta}(re^{2\pi it}) = \begin{cases} re^{2 \pi i dt} & re^{2\pi i t} \in S_d(\theta), t \neq d^{k-2}\theta\\
            (r-1)e^{2 \pi i dt} & r\geq 1, t = d^{k-2}\theta\\
            (r+1)e^{2\pi i \theta} &r\geq 0, t\in \mu_d^{-1}(\theta)  \cup \{\frac{\theta+i+ \epsilon}{d}\}\\            \infty &r=\infty
            \end{cases}
        \end{align*}
    \item[-] Now suppose $\theta$ is periodic under $\mu_d$, and $j=j_\theta$. Let $\omega = \frac{\theta+j_\theta}{d}$. If $\theta$ is the smaller angle in a companion pair, let $\alpha = \frac{\theta+j_\theta+\epsilon}{d}$. In all other cases, let $\alpha = \frac{\theta+j_\theta-\epsilon}{d}$. Define the tweaked spider $\widetilde{S}^{j,\epsilon}_d(\theta)$ as follows.
                \begin{align*}
                    \widetilde{S}^{j,\epsilon}_d(\theta) : = \widetilde{S}_d(\theta) \cup  \Big\{re^{2\pi i\alpha}: 0\leq r \leq \infty\Big\}
                \end{align*}
                
                We construct $\mathcal{G}_{d,j_\theta,\theta}: \widetilde{S}_d^{j,\epsilon}(\theta) \longrightarrow \widetilde{S}_d^{j,\epsilon}(\theta)$ as follows:
        \begin{align*}
            \mathcal{G}_{d,j,\theta}(re^{2 \pi it}) = \begin{cases} re^{2 \pi i dt} & re^{2 \pi it} \in S_d(\theta) , t \neq d^{k-2}\theta\\
            (r-1)e^{2 \pi i\omega} & r\geq 1, t= d^{k-2}\theta\\
            (r+1)e^{2\pi i \theta} &r\geq 0, t \in \mu_d^{-1}(\theta) \cup \{\alpha\} \\
            \infty &r=\infty
            \end{cases}
        \end{align*}
\end{itemize}
\begin{figure}[t]
    \centering
    \includegraphics[scale=0.4]{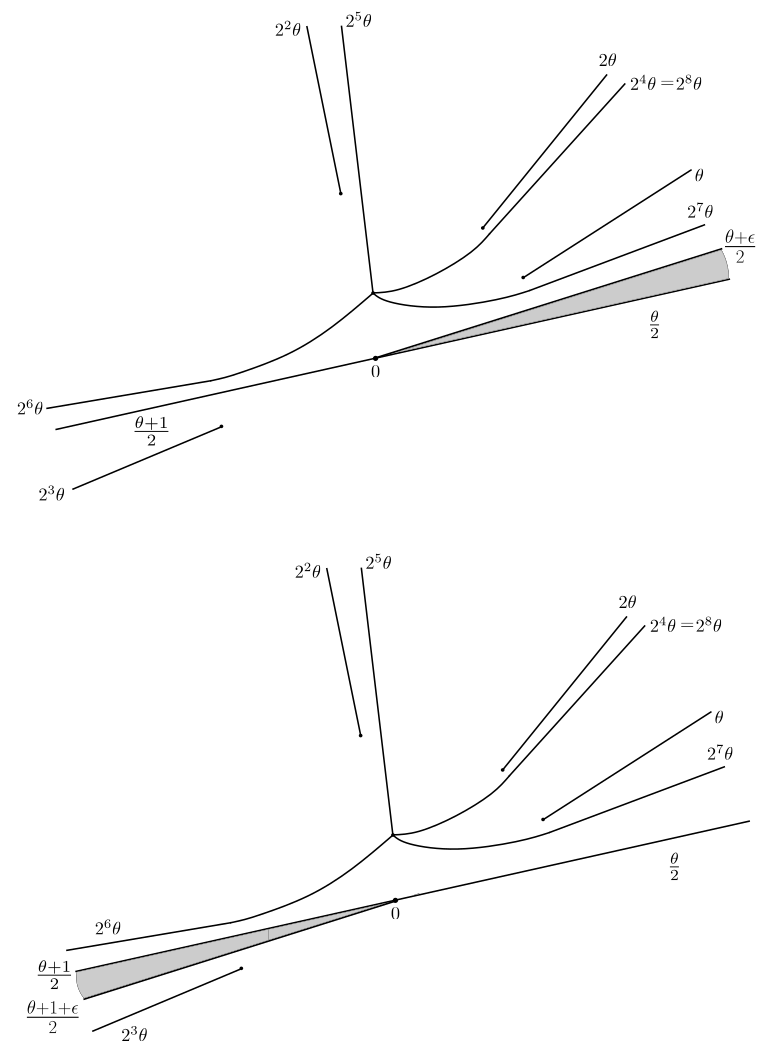}
    \caption{$[\widetilde{S}_2^{0}(\theta)]$ and $[\widetilde{S}_2^{1}(\theta)]$ for $\theta = \frac{17}{2^4(2^4-1)}$}
    \label{fig:tweaked spider}
\end{figure}
        In both cases, the map $\mathcal{G}_{d,j,\theta}$ can be extended using the Alexander trick to a Thurston polynomial of degree $d+1$, with critical set $\{0,\infty\}$. By Proposition~\ref{prop:genspider}, The Thurston class of this extension is independent of $\epsilon$, so we let $\widetilde{S}^{j}_d(\theta)$ denote  $\widetilde{S}^{j,\epsilon}_d(\theta)$, and $S^j_d(\theta) = \mathcal{G}_{d,j,\theta}(\widetilde{S}^j_d(\theta))$. 
        
        If $\theta$ is periodic under $\mu_d$ with exact period $k$, then $0$ periodic under $\mathcal{G}_{d,j,\theta}$ with exact period $k$.  If  $\theta$ has preperiod $\ell$ and period $k$, then $0$ has preperiod $\ell$ and period $k$ under $\mu_d$.

        The above construction of $\mathcal{G}_{d,j,\theta}$ can be equivalently defined in terms of an arc blow-up of $f_{d,\theta}$: we first slit open the leg $L_j = \Big\{re^{\frac{2\pi i (\theta+j)}{d}}\Big\}$ of $\widetilde{S}_d(\theta)$, and glue to this slit a copy of the slit plane $\hat{\C} \setminus \R_{\leq 0}$ (each copy of $\R_{\leq 0}$ is identified with a copy of $L_j$). We now map this newly attached slit plane onto the sphere (see  \cite{pilgrimlei} where arc blow-ups among external rays was first defined). 
\begin{proposition}
\label{prop:addsector=angle}
      \begin{enumerate}
      \item If $\theta$ is $k-$periodic under $\mu_d$, then $\phi_k = \frac{\phi+u_{j,\theta}(\theta_k)}{d+1}$;
      \item  There exists a homeomorphism $h\colon \widetilde{S}^j_d(\theta) \longrightarrow \widetilde{S}_{d+1}(\phi)$ that satisfies 
      \begin{itemize}
          \item $h(0)=0$
          \item $h(\infty) = \infty$
          \item $h(S^j_d(\theta)) = S_{d+1}(\phi)$
      \end{itemize}
     while preserving the circular order of legs at $\infty$, and conjugates $\mathcal{G}_{d,j,\theta}$ to $\mathcal{F}_{d+1,\phi}$. 
      \item $\mathcal{G}_{d,j,\theta}$ is Thurston equivalent to $\mathcal{F}_{d+1,\phi}$.
      \end{enumerate}
\end{proposition}    
\begin{proof}
\begin{enumerate}
 \item is clear by the definition of $\phi$.
          \item  For $re^{2 \pi i t} \in \widetilde{S}^j_d(\theta)$, let
\begin{align*}
    h(re^{2\pi it})&= \begin{cases} re^{2\pi i \phi_n}&t=\theta_n\text{ for some }n\geq 1\\
    re^{2 \pi i (\frac{\phi+m}{d+1})}&t= \frac{\theta+m}{d},m<j\\
    re^{2 \pi i (\frac{\phi+m+1}{d+1})}&t= \frac{\theta+m}{d},m>j\\
        re^{2 \pi i (\frac{\phi+j}{d+1})}&t= \frac{\theta+j-\epsilon}{d}\\
                re^{2 \pi i (\frac{\phi+j+1}{d})}&t= \frac{\theta+j+\epsilon}{d}
    \end{cases}
\end{align*}
It is easy to verify that $h$ satisfies the required properties.
\item 
Clear by (2) and Proposition~\ref{prop:genspider}.
\end{enumerate}
\end{proof}    
\begin{remark}\label{remark:kneadpres}     

    Let $\nu$ be the kneading sequence of $\theta$. 
Suppose $T^{stat}_{d,m}(\theta)$ is the sector containing $\theta$ in its interior. Construct a sequence  $\nu^j$ using the following rule:
\begin{itemize}
    \item Every symbol $\nu_n$ corresponding to $\theta_n$ in one of the  sectors  $T^{stat}_{d,m}(\theta),T^{stat}_{d,m+1}(\theta),...,T^{stat}_{d,j-1}(\theta)$ in counterclockwise order - that is, a symbol in $\{0,1,...,j-1-m\}$,  stays the same. Formally, $$\nu_n  \in \{0,1,...,j-1-m\} \implies \nu^j_n = \nu_n$$
    \item Every symbol $\nu_n$, corresponding to $\theta_n$ in one of the sectors $T^{stat}_{d,j}(\theta),T^{stat}_{d,j+1}(\theta),...,T^{stat}_{d,m-1}(\theta)$ in counterclockwise order - that is, a symbol in $\{j-m,j+1-m,...,d-1\}$ increases by $1$ modulo $d$. That is,
    $$\nu_n \in \{j-m, j+1-m,d-1\} \implies \nu^j_n = \nu_n +1 \text{ (mod } d)$$
    \item If $\nu_n = *$, let $\nu^j_n = *$
\end{itemize}
In the resulting sequence $\nu^j$,  $\nu^j_m \not = \nu^j_n$ unless $\nu_m = \nu_n$. Proposition~\ref{prop:monotonedyn} shows that $\nu^j$ is the kneading sequence of $Z_j(\theta)$.

We refer here to \cite{Kaffl2006}, where the author shows that preperiodic and $*-$periodic kneading sequences of any degree $d$ can be embedded into the same collection in degree $d' \geq d$ in $\binom{d'-1}{d-1}$ ways. The above discussion illustrates this result for $d'=d+1$ where we restrict to the set of kneading sequences realizable by polynomial maps.
\end{remark} 

\subsection{Image of $Z_j$}
Fix $j \in \{0,1,...,d-1\}$.
\begin{proposition}\label{prop:maximalityprep}Let $\phi \in \Q/\Z$. 
If the  $\mu_{d+1}-$orbit of $\phi$ does not intersect $T^{stat}_{d+1,j}(\phi)$, then $\phi = Z_j(\theta)$ for a unique angle $\theta \in \Q/\Z$.
\end{proposition}
\begin{proof}
If $\phi$ has preperiod $0$ and period $1$, $\phi = Z_j(\theta)$ for some $\theta$ of the form $\frac{m}{d-1}$. 

Otherwise, let $\ell,k$ be the pre-period and period respectively of $\phi$, and choose a $d+1$ - adic expansion $.y_1y_2....y_\ell\overline{y_{\ell+1}y_{\ell+2}...y_{k+\ell}}$ for $\phi$, that does not terminate in a constant stream of $d$'s if $k=1$. 

We define
\begin{align*}
        w:\mathscr{O}_{d+1}(\varphi)&\longrightarrow \{0,1,...,d-1\}\\
        w(t)& = \begin{cases}m & m \in \{0,1,..,j-1\}, t \in [\frac{m}{d+1},\frac{m+1}{d+1})\\
        j &t \in [\frac{j}{d+1},\frac{\phi+j}{d+1}]\\
        j &t \in [\frac{\phi+j+1}{d+1},\frac{j+2}{d+1})\\
        m-1& m \in \{j+2,j+3,...,d\}, t \in [\frac{m}{d+1},\frac{m+1}{d+1})\end{cases}
    \end{align*}

We claim that $\theta := \sum_{n=1}^\infty \frac{w(\phi_n)}{d^n}$ is preperiodic under $\mu_d$ with preperiod $\ell$ and period $k$, and that $Z_{j}(\theta) = \phi$. 

    Let $x_n = w(\phi_n),$ $\theta_n = d^{n-1}\theta$, and $m,j$ be the preperiod of $\theta$ under $\mu_d$. It is clear that $m\leq \ell$  and that $j|k$.
     
        If $j<k$,  then  $\theta_{m+1} = \theta_{m+j+1}$, which can happen only if $x_n = x_{n+j}$  for all $n>m$. 
We note:
\begin{align*}
    \phi_n \in \Big[0,\frac{\phi+j}{d+1}\Biggl] & \implies x_n = y_n\\
    \phi_n \in \Biggl[\frac{\phi+j+1}{d+1},1\Biggl) &\implies x_n = y_n-1
\end{align*}
So the only way that $x_n = x_{n+j}$ is if $y_n = y_{n+j}$ or $\{y_n, y_{n+j}\} = \{j,j+1\}$. 

    Let $D = \{n>m| y_{n} \neq y_{n+j}\}$. Note that $k \in D$ - in particular, $D$ is nonempty. Let $r$ be the least element in $D$. If $y_{r}<y_{r+j}$, this means that $\phi_r \in (\frac{j}{d+1},\frac{\phi+j}{d+1}]$, and that $\phi_{r+j} \in [\frac{\phi+j+1}{d+1},\frac{j+2}{d+1})$. We also note that $\phi_r = \frac{\phi+j}{d}$ and $\phi_{r+j} = \frac{\phi+j+1}{d}$ cannot be simultaneously true, since only one angle of the form $\frac{\phi+m}{d}$ can belong to $\mathscr{O}_{d+1}(\phi)$. Therefore,  
    \begin{align*}
        \phi_{r+1} &< \phi_{r+j+1}\\
        .y_{r+1}y_{r+2}..... &< .y_{r+j+1}y_{r+j+2}.....
    \end{align*}
    So at the next $n>r$ in $D$, we must have $y_{n} < y_{n+j}$.
            The above discussion shows that exactly one of the following is true:
            \begin{enumerate}
                \item $\forall n >m $, $y_{n} \leq y_{n+j}$
                \item $\forall n>m$, $y_{n}\geq y_{n+j}$
            \end{enumerate}
            Suppose (1) is true, we then have, for any $n \in D$,
            \begin{align*}
                y_{n}\leq y_{n+j} \leq ....\leq y_{n+(\frac{k}{j}-1)j} \leq y_{n+k} = y_{n}
            \end{align*}
            That is, $y_{n} = y_{n+j}$, which contradicts the fact that $\phi$ has eventual period $k$. We get a similar contradiction for (2), and therefore, we must have $j=k$. 
            
Suppose  $m<\ell$,  write $m$ as $\ell - rk$ for some integer $r$. Then 
        \begin{align*}
            \theta_{m+1} &= \theta_{k+m+1}
        \end{align*}
and this holds only if for all  $n>m$, we have $x_n = x_{n+k}$. But this means, for all  $n>m$ such that $y_n \neq y_{n+k}$, we have $y_n , y_{n+k} \in \{j,j+1\}$. 

We note that $y_\ell \neq y_{k+\ell}$. Without loss of generality, suppose, $y_\ell = j$, $y_{k+\ell} = j+1$. Then we have $\phi_\ell \in \big[\frac{j}{d+1},\frac{\phi+j}{d+1}\big]$ and $\phi_{k+\ell} \in \big[\frac{\phi+j}{d+1},\frac{j+2}{d+1}\big)$.\\ \\
Therefore, 
$$\phi_{\ell+1}<\phi<\phi_{k+\ell+1}$$
which is a contradiction.
Thus, we must have $m=\ell$.\\
It is now clear that $Z_{j}(\theta)=\phi$. By Proposition~ \ref{prop:monotone}, $\theta$ is unique.
\end{proof}
\begin{proof}[Proof of Lemma~\ref{lemma:injective_maps_of_angles}]
    By Propositions~\ref{prop:monotone} and Propositions~\ref{prop:missedsector}, Lemma~\ref{lemma:injective_maps_of_angles} is clear. 
\end{proof}
\section{Proof of Lemma~\ref{lemma:combi_embeddings}}\label{sec:proofoflemma1.2}
Fix a degree $d\geq 2$. 
For $j=0,1,...,d-1$, we will construct combinatorial embeddings $\mathscr{E}_j:\mathcal{P}_d \longrightarrow \mathcal{P}_{d+1}$ using Lemma~\ref{lemma:injective_maps_of_angles}, and prove Lemma~\ref{lemma:combi_embeddings} by exhibiting additional properties of the $Z_j$'s.

Fix $j \in \{0,1,...,d-1\}$ and $\lambda \in \mathcal{P}_d$. \subsection{Definition for critically periodic parameters}\label{sec:periodic}
First suppose $\lambda$ has a $k-$periodic critical point, with $k\geq 2$, and choose $c \in M_d(\lambda)$.

By {\cite[Theorem~10.3.9]{teichmuller_theory_vol2}}, for any $\theta \in \Omega_d(c)$,  the map $\mathcal{G}_{d,j,\theta}$ defined in Section~\ref{sec:preserving spiders} is Thurston equivalent to a polynomial $\lambda_{j,\theta}\big(1+\frac{z}{d+1}\big)^{d+1}$ such that $\phi=Z_j(\theta) \in \Theta_{d+1}(\lambda_{j,\theta})$. We will show that $\lambda_{j,\theta}$ is independent of the choice of $\theta$ and $c$, allowing us to define $\mathscr{E}_j(\lambda) = \lambda_{j,\theta}$. 
For the rest of Section~\ref{sec:periodic}, fix $\theta \in \Omega_d(c)$. Note that the parameter ray $R_d(\theta)$ could land either at the root or a co-root of the hyperbolic component $U$ containing $c$. 
\subsubsection{Landing at a root}
We first suppose that $\theta$ lands at the root of $U$, and has companion $\theta'$. Let $\varphi = Z_j(\theta)$ and $\varphi' = Z_j(\theta')$. 
\begin{proposition}
\label{prop:companion}
      $\phi, \phi'$ are companion angles under $\mu_{d+1}$.
\end{proposition}
\begin{proof}
    Let $\mathscr{A} = \{\mathscr{A}_1,...,\mathscr{A}_r\}$ be the orbit portrait generated by $(\theta,\theta')$, and let $\mathscr{O}(\theta,\theta') = \bigcup_{i=1}^r \mathscr{A}_i$.
    
    The symbol shift functions $u_{j,\theta}$ and $u_{j,\theta'}$  coincide on $\mathscr{O}(\theta, \theta')$, since $\mathscr{O}(\theta,\theta')\cap \big(\frac{\theta+j}{d},\frac{\theta'+j}{d}\big) = \emptyset$.
    We define the sets $\mathscr{B}_i$ as follows:\begin{align*}
    \mathscr{B}_i &= \{\phi_n: \theta_n  \in \mathscr{A}_i\} \cup \{\phi_n ': \theta_n ' \in \mathscr{A}_i\}
    \end{align*}
\noindent  The collection $\mathscr{B} = \{\mathscr{B}_1,...,\mathscr{B}_r\}$ is a  partition of the union of the orbits of $\phi$ and $\phi'$. In order to show that $(\phi,\phi')$ is a companion pair, we first show that the angles in their orbits taken together form a formal orbit portrait.\\
Let 
\begin{align*}
    \mathscr{B}(\theta,\theta') = \bigcup_{j=1}^r \mathscr{B}_j
\end{align*}
In order to show that $\mathscr{B}$ is a formal orbit portrait, we will prove that it satisfies all properties in Definition~\ref{defn:formal_portrait}.
\begin{enumerate}
\item is clear by definition.
\item The fact that $\mathscr{B}_i$ maps bijectively onto $\mathscr{B}_{i+1}$,  is clear since $\mathscr{A}_i$ maps bijectively onto $\mathscr{A}_{i+1}$ under $\mu_d$. The rest follows from Proposition~ \ref{prop:spiderconjugacy}.
\item Define 
\begin{align*}
    \mathscr{B}_i(\theta) & = \mathscr{O}_{d+1}(\phi) \cap \mathscr{B}_i\\ \mathscr{B}_i(\theta') & = \mathscr{O}_{d+1}(\phi') \cap \mathscr{B}_i
\end{align*}
Since $\mathscr{A}_i \subset T^{stat}_{d,m}(\theta)$ for some $m \in \{0,1,...,d-1\}$, we have $\mathscr{B}_i(\theta) \subset T^{stat}_{d+1,n}(\phi)$  and $\mathscr{B}_i(\theta') \subset T^{stat}_{d+1,n}(\phi')$ for some $n \in \{0,1,..,d\}$. It suffices to show that 
\begin{align*}
    \mathscr{B}_i \subset \Biggl(\frac{\phi'+n}{d+1},\frac{\phi+n+1}{d+1}\Biggl) = T^{stat}_{d+1,n}(\phi) \cap T^{stat}_{d+1,n}(\phi')
\end{align*}
For any $s \in \{0,1,...,d\}$, if there exists  $\psi \in \mathscr{B}_i \cap \big[\frac{\phi+s}{d+1},\frac{\phi'+s}{d+1}\big]$, then there exists $\alpha \in \mathscr{A}_i \cap \big[\frac{\theta+r}{d},\frac{\theta'+r}{d}\big]$ for some $r$, which is not possible. Thus, for all $ s \in \{0,1,...,d\}$, 
\begin{align*}
    \mathscr{B}_i \cap \Biggl[\frac{\phi+s}{d+1},\frac{\phi'+s}{d+1}\Biggl]=\emptyset
\end{align*}
The result follows immediately.
\item is clear, since the period of all angles in $\mathscr{B}(\theta,\theta')$ is equal to $k$, and since $\mathscr{A}$ is a formal orbit portrait,  $k = rp$ for some $p\geq 1$.
\item 
If $\mathscr{A}_n \subset \big(\frac{\theta'+j}{d},\frac{\theta+j+1}{d}\big)$, then $\mathscr{B}_n \subset \big(\frac{\phi'+j}{d+1},\frac{\phi+j+1}{d+1}\big)$.

Given distinct integers $0 \leq i, i' \leq d$, by property (2),  $\mathscr{B}_{n,i}$ and $\mathscr{B}_{n,i'}$ are unlinked. Given $n \neq m$, and any $i,i'$, we want to show that $\mathscr{B}_{n,i}, \mathscr{B}_{m,i'}$ are unlinked.
             If $i=i'=0$, then this follows from the fact that $\mathscr{A}_n, \mathscr{A}_{m}$ are unlinked. If at least one of $i,i'$ is nonzero and $\mathscr{B}_{n,i}$ and $\mathscr{B}_{m,i'}$ are linked, then without loss of generality, we can find angles $\alpha, \beta \in \mathscr{B}_{n,i}$ and $\eta, \delta \in \mathscr{B}_{m,i'}$ that satisfy
             \begin{align*}
                 \alpha < \eta < \beta < \delta
             \end{align*}
             This implies 
             \begin{align*}
                (d+1)\alpha < (d+1)\eta < (d+1)\beta < (d+1)\delta
             \end{align*}
              The inequalities are strict since $\mathscr{B}_{n+1}$ and $\mathscr{B}_{m+1}$ are disjoint. But  $\mathscr{B}_s$ maps bijectively onto $\mathscr{B}_{s+1}$ for each $s$, and the above inequality implies that $\mathscr{B}_{n+1}, \mathscr{B}_{m+1}$ are linked, which is a contradiction.
\end{enumerate}
Thus $\mathscr{B}$ is a formal, non-trivial orbit portrait. As a last step, we show that $\phi,\phi'$ are the characteristic angles of this orbit portrait.

The interval $(\phi, \phi')$ has either the smallest or largest arc length in $\mathscr{B}_1$. We let $\ell_i$ be the length of the unique complementary arc $\gamma_i$ of $\mathscr{B}_i$ of length greater than $\frac{1}{d}$. 
    Each $\gamma_i$ is a critical arc - that is, under multiplication by $d+1$, it covers the circle $d$ times. Furthermore, the $\gamma_i$'s are the only critical arcs of $\mathscr{B}$.\\ \\
    For $i\neq r$, $\gamma_i$ is strictly contained in  $\R/\Z \setminus \big[\frac{\phi'+m}{d+1},\frac{\phi+m+1}{d+1}]$ for some $m$, whereas $\gamma_r =\R/\Z \setminus  \big[\frac{\phi'+m_0}{d+1},\frac{\phi+m_0+1}{d+1}\big]$ for some $m_0$.  This proves that $\ell_r = \max_i \ell_i$.

    Therefore, the critical value arc bounded by $\mu_{d+1}(\partial \gamma_r)$ is the shortest critical value arc among all critical value arcs. But we note that $\mu_{d+1}(\partial \gamma_r) = \{\phi,\phi'\}$, meaning that 
    \begin{align*}
        d(\phi',\phi) = \min_{i}\min_{\alpha,\beta \in \mathscr{B}_i} d(\alpha,\beta)
    \end{align*}
    where $d(\alpha,\beta)$ is the length of the smaller arc in $\R/\Z \setminus \{\alpha,\beta\}$. 
    This is the definition of the characteristic angle pair, and the result follows.
\end{proof}
\begin{proposition}\label{prop:maximalityper}
Given a companion pair $(\phi,\phi')$ landing at a hyperbolic root in $\mathcal{M}_{d+1}$ such that the orbit of $\phi$ (or $\phi'$) under $\mu_{d+1}$ does not intersect the interior of $T^{stat}_{d+1,j}(\phi)$, there exists a companion pair $(\alpha,\alpha')$ periodic under $\mu_d$ such that $\phi = Z_j(\alpha), \phi' = Z_j(\alpha')$.
\end{proposition}
\begin{proof}
    We apply Proposition~\ref{prop:maximalityprep} to $\phi$ to get an angle $\alpha$ with preperiod $0$ and period equal to that of $\phi$, and $Z_j(\alpha) = \phi$. If $\alpha$ lands at a co-root in $\mathcal{M}_d$, this would imply that $\phi$ lands at a co-root in $\mathcal{M}_{d+1}$, which contradicts our assumption. Thus $\alpha$ lands at a root, and has a companion $\alpha'$. By Proposition~\ref{prop:companion}, $\phi' = Z_j(\alpha')$.
\end{proof}
\begin{proposition}\label{ind_monic}
Let $c'= e^{\frac{2 \pi i n}{d-1}}c$ for some $n \in \{0,1,...,d-1\}$. For $\theta' = \theta + \frac{n}{d-1}$, we  have 
\begin{align*}
 \lambda_{j,\theta} = \lambda_{j,\theta'}   
\end{align*}
\end{proposition}
\begin{proof}
$Z_j(\theta)$ is an angular coordinate for a parameter $c_{j,\theta} \in M_{d+1}(\lambda_{j,\theta})$. By Proposition~\ref{prop:symmetrypres}, $\phi':=Z_j(\theta')= \phi+\frac{n}{d}$, and hence $\phi' \in \Theta_{d+1}(\lambda_{j,\theta})$. But  $\phi'$ is an angular coordinate for $e^{\frac{2\pi i n}{d}}c_{j,\theta}\in M_{d+1}(\lambda_{j,\theta'})$, and thus the statement follows. 
\end{proof}
\begin{example}
Let $d=2$.

The quadratic rabbit polynomial (so called because its Julia set looks like a rabbit) is given by  $z^2+c$ where $c\approx-0.122561+0.744862i$. The root of the hyperbolic component containing the rabbit is the landing point of the companion angles $(\theta,\theta') = (\frac{1}{7},\frac{2}{7})$.

The pair  $(Z_0(\theta),Z_0(\theta')) = (\frac{14}{26},\frac{16}{26})$ is a companion pair forming angular coordinates for a cubic rabbit parameter $c_{0,\theta} = c_{0,\theta'} \approx -0.535 - 0.554999i$.

\end{example}
\subsubsection{Landing at a co-root}\label{sec:coroots}
In this section, we assume that $\theta$ lands at a co-root of $U$. There exists an angle pair $(\alpha,\alpha')$ with period $k$ lands at the root of $U$. By the previous section, the angles $\psi = Z_j(\alpha), \psi'=Z_j(\alpha')$ land at the root of a hyperbolic component $V \subset \mathcal{M}_{d+1}$. Since $Z_j$ is order-preserving, $\phi = Z_j(\theta)$ lands in the wake of $(\psi,\psi')$, and by  Proposition~\ref{prop:monotonedyn}, $\phi, \psi$ and $ \psi' $ all have  period $k$ under $\mu_{d+1}$. Additionally, by Remark~\ref{remark:kneadpres}, the kneading sequences of $\phi, \psi $ and $\psi'$ coincide with $\nu^j$ up to and including the index $k-1$. The itineraries of $\phi$ with respect to both $\psi$ and $\psi'$ differ from each other, and from $\nu^j$ at the $k$th position. Since $\psi < \phi <\psi'$, we have $\phi_{k} \in \big(\frac{\psi+n}{d+1},\frac{\psi'+n}{d+1}\big)$ for some $n \in \{0,1,...,d\}$. 

We will show that $\phi$ lands at a co-root of $V$.
\begin{proposition}\label{prop:corootexists}
\begin{enumerate}
    \item $\psi_k \neq \frac{\psi+n}{d+1}$, and $\psi'_k \neq \frac{\psi'+n}{d+1}$
    \item There exists an angle $\phi'$ landing at a co-root of $V$ such that $\phi'_k \in (\frac{\psi+n}{d+1},\frac{\psi'+n}{d+1})$. 
\end{enumerate}
\end{proposition}
\begin{proof}
\begin{enumerate}
    \item Note that $d(\psi,\psi') < d(\psi_2,\psi'_2) < .....< d(\psi_{k},\psi'_{k})$. 
    This also implies $d(\psi,\phi) < d(\psi_2,\phi_2) < .....< d(\psi_{k},\phi_{k})$. Suppose 
    $\psi_k = \frac{\psi+n}{d+1}$, we would then have $d(\psi_k,\phi_k) = \frac{d(\psi,\phi)}{n}$, which is a contradiction. By a similar argument, $\psi'_k \neq \frac{\psi'+n}{d+1}$. 
    \item For any co-root angle $\phi'$, $\phi'_k$ cannot be in $[\frac{\psi'+m}{d+1},\frac{\psi+m+1}{d+1}]$ for any $m$, since this would imply that $\phi' \not \in (\psi,\psi')$.  Therefore each $\phi'_k$ belongs to $(\frac{\psi+m}{d+1},\frac{\psi'+m}{d+1})$, with $\psi_k \neq \frac{\psi+m}{d+1}$ and $\psi'_k \neq \frac{\psi'+m}{d+1}$, by (1). There are $d-1$ co-roots and $d-1$ values of $m$ that satisfy this property. Suppose we have two co-root angles $\phi',\phi''$ with $\phi'_k,\phi''_k \in (\frac{\psi+m}{d+1},\frac{\psi'+m}{d+1})$, this again contradicts the chain of inequalities given in (1), thus for each $m$ with  $\psi_k \neq \frac{\psi+m}{d+1}$ and $\psi'_k \neq \frac{\psi'+m}{d+1}$, there exists exactly one co-root angle $\phi' \in (\frac{\psi+m}{d+1},\frac{\psi'+m}{d+1})$
\end{enumerate}
\end{proof}
The angles $\phi, \phi'$ thus have the same itinerary with respect to $\psi$, and therefore, the dynamic rays at angles $\phi$ and $\phi'$ land at the same point $z_0$ in the plane of $f_{{d+1},\widetilde{c}}$, where $\widetilde{c}$ is the center of $V$, making $z_0$ a cut point in the Julia set of $f_{{d+1},\widetilde{c}}$ unless $\phi = \phi'$. But co-root angles do not land at cut-points, forcing $\phi = \phi'$.
 
The discussion in this section shows that for a periodic parameter $\lambda$, we may define $\mathscr{E}_i(\lambda) = \lambda_{i,\theta}$ for any angle $\theta \in \Omega_d(c)$, for any choice of $c \in M_d(\lambda)$.

\subsection{Definition for critically preperiodic parameters}\label{sec:algorithmangleprep}
Supposing $\lambda$ has a strictly preperiodic critical point, choose $c \in M_d(\lambda)$ and $\theta \in \Omega_d(c)$. Let us denote by $c_{j,\theta}$ the landing point of  $\phi = Z_j(\theta)$ in $\mathcal{M}_{d+1}$. Since $\phi$ is strictly preperiodic under $\mu_{d+1}$, $c_{j,\theta}$ is critically preperiodic.

\begin{proposition}\label{prop:branchpres}For $\theta'\in \Omega_d(c)$,  both $\phi$ and $\phi' = Z_j(\theta')$ land at $c_{j,\theta}$. \end{proposition}
\begin{proof}
Let $\gamma$ denote the Carath\'{e}odory loop of $f_{d,c}$. Then, given any $\alpha \in \Omega_d(c)$ and $t_1, t_2 \in \R/\Z$,
\begin{align*}
\gamma(t_1) = \gamma(t_2) \iff \Sigma_{d,\alpha}(t_1) = \Sigma_{d,\alpha}(t_2) 
\end{align*}

This gives us  $\Sigma_{d,\theta}(\theta') = \Sigma_{d,\theta}(\theta) = \Sigma_{d,\theta'}(\theta) = \Sigma_{d,\theta'}(\theta')$. This common sequence is also the kneading sequence of both $\theta$ and $\theta'$. Let us call it $\nu$.
By Remark~\ref{remark:kneadpres}, the angles $\phi$ and $\phi'$ have the kneading sequence $\nu^j$, and by Proposition~\ref{prop:monotonedyn}, it is clear that $\Sigma_{d+1,\phi}(\phi) =\Sigma_{d+1,\phi}(\phi') = \Sigma_{d+1,\phi'}(\phi') = \Sigma_{d+1,\phi'}(\phi) =  \nu^j$. 
So in the dynamical plane of $z^{d+1}+c_{j,\theta}$,  $c_{j,\theta}=\gamma_j (\phi) = \gamma_j(\phi')$, where $\gamma_j$ is the corresponding Carath\'{e}odory loop. Hence the parameter ray at $\phi'$ also lands at $c_{j,\theta}$. 
\end{proof}
This proof shows that $c_{j,\theta}$ depends only on $c$. Proposition~\ref{ind_monic} also applies to preperiodic parameters, hence we may define $\mathscr{E}_j(\lambda) = \lambda_{j,\theta} = (d+1)c_{j,\theta}^d$ for any $c \in M_d(c)$ and any $\theta \in \Omega_d(c)$. 
\begin{example}
Set $d=2$. The angles $\theta = \frac{17}{2^4(2^4-1)}$,  $\theta' = \frac{19}{2^4(2^4-1)}$ land at $c \approx 0.63571 +0.59124 i \in \mathcal{M}_2$. 

The angles $Z_1(\theta) = \frac{163}{3^4(3^4-1)}$, $Z_1(\theta') = \frac{169}{3^3(3^4-1)}$ land at $c_{1,\theta} = c_{1,\theta'} \approx 0.62745+0.29882i.
 \in \mathcal{M}_3$.
\end{example}
This finishes the definition of $\mathscr{E}_i(\lambda)$ for all $\lambda \in \mathcal{P}_d$. In the rest of the section, 
we show that $\mathscr{E}_j$ is a combinatorial embedding. 

As in previous sections, fix $\lambda \in \mathcal{P}_d$ and, $c \in M_d(\lambda)$ and $ j \in \{0,1,...d-1\}$. 
\subsection{Properties of $\mathscr{E}_j$}

First suppose $\lambda$ is a periodic parameter with critical value of period $k\geq 2$, and $c$ is the center of hyperbolic component $U$. 
\begin{proposition}\label{prop:addresspres}\label{prop:primpres} Suppose the angle $\theta$ lands at the root of $U$ and has companion $\theta'$.
         Let $V$ be the hyperbolic component in $\mathcal{M}_{d+1}$ bounded by $(Z_j(\theta),Z_j(\theta'))$. Then $U$ is primitive if and only if $V$ is primitive.
\end{proposition}
\begin{proof}
        Given any sequence $*$-periodic sequence $\nu$, consider the function \begin{align*}
        \rho_\nu: \N \cup \{\infty\} &\longrightarrow \N \cup \{\infty\}\\
        \rho_\nu(k)& = \begin{cases}\inf\{n>k: \nu_{n-k}\neq \nu_n\} &k \neq \infty\\
        \infty &k = \infty\end{cases}
    \end{align*}\\
    From \cite{bruinschleicher}, we know that for any hyperbolic component $V$ in $\mathcal{M}_d$ with kneading sequence $\mu$, the internal address of $\nu$ can be computed as $orb_{\rho_\mu}(1)$, upto the entry before $\infty$. Let $\nu$ be the kneading sequence of $\theta$. Clearly, $\rho_\nu = \rho_{\nu^j}$. 
    
A hyperbolic component is a satellite if and only if the penultimate entry of its internal address divides the final entry (see \cite{shleicherintaddress}). From this, the second statement follows. 
  \end{proof}
\begin{proposition}\label{prop:satpres}
If $(\theta,\theta')$ is a satellite of $(\psi,\psi')$, then $(Z_j(\theta),Z_j(\theta'))$ is a satellite of $(Z_j(\psi), Z_j(\psi'))$.
\end{proposition}
\begin{proof}
We will prove something stronger.

Let $k$ be the period of $\psi$, and let $\nu$ be the kneading sequence of $(\theta,\theta')$.
$(Z_j(\theta), Z_j(\theta'))$ is a satellite of a ray pair $(\alpha, \alpha')$ of period $k$ (since its internal address is the same as that of $(\theta,\theta')$), and it lies in the wake of $(Z_j(\psi),Z_j(\psi'))$. 
Now suppose that the internal address of $(\theta,\theta')$ is given by   $$1 \mapsto s_1 \mapsto ... \mapsto s_{m-1} = k \mapsto s_{m} = n$$

Then there exist ray pairs $R_j, P_j $ periodic under $\mu_{d},\mu_{d+1}$ respectively of exact period $l$ that correspond to the entry $s_j$, and moreover, we have 
\begin{align*}
    R_m &= (\theta, \theta')\\
    R_{m-1} &= (\psi,\psi')\\
    P_m &= (Z_j(\theta), Z_j(\theta'))\\
    P_{m-1} &= (\alpha, \alpha')
\end{align*}
If $R = (\alpha',\beta')$, let $Z_j(R)$ denote $(Z_j(\alpha'),Z_j(\beta'))$.

For each $j$, $P_m$ is in the wake of $Z_j(R_j)$. 
Additionally, by definition, 
\begin{align*}
    P_1 = Z_j(R_1)\\
    P_m = Z_j(R_m)
\end{align*}
Let $t>0$ be the least index where $Z(R_{t}) \neq P_{t}$. By Proposition~\ref{prop:monotone} , $P_m$ is in the wake of both $P_{t}$ and $Z_j(R_{t})$. Therefore, one of the following is true: either $P_{t}$ is in the wake of $Z_j(R_{t})$ or $Z(R_{t})$ is in the wake of $P_{t}$. In either case, by a theorem of Lavaurs (\cite{lavaurs}), there exists a ray pair $R$ of period $p<s_{t}$ that separates $P_{t}$ and $ Z_j(R_{t})$. \\ \\
But $R$ lies in the wake of $P_{t-1}$, and $P_m$ lies in the wake of $R$. This suggests that $\rho_\nu(s_{t-1}) \leq p < s_{t}$, which is a contradiction to the fact that $\rho_\nu(s_{t-1}) = s_{t}$. Therefore, $\forall $ indices $j$, we have $P_j = Z_j(R_j)$. Particularly, $(\alpha,\alpha') = P_{m-1} = (Z_j(\theta),Z_j(\theta'))$.
\end{proof}
\subsubsection{More on co-roots}
    Let $\theta$ be an angle that lands at a co-root of $U$, whose root angles are $(\alpha,\alpha')$. Let $V$ be the component in $\mathcal{M}_{d+1}$ on which $Z_j(\alpha)$ lands.
    
    Note that 
    \begin{align*}
        Z_j(\theta) = .\overline{u_{j,\theta}(\theta_1)u_{j,\theta}(\theta_2)....u_{j,\theta}(\theta_k)j+1}
    \end{align*}
    Let $\psi = Z_j(\theta)$, and $\psi' =  .\overline{u_{j,\theta}(\theta_1)u_{j,\theta}(\theta_2)....u_{j,\theta}(\theta_k)j}$.
 
   From Section~\ref{sec:proofoflemma1.2}, we know that $\psi$ lands at a co-root of $V$. Note that $\psi'$ has itinerary $\nu^j$ with respect to $Z_j(\alpha)$. Also note that $\psi' = \psi - \frac{1}{(d+1)^k-1}$. By  Proposition~\ref{prop:corootexists}, we can show that there exists an angle $\psi''$ landing at a co-root of $V$ whose itinerary with respect to $Z_j(\alpha)$ coincides with that of $\psi'$. This forces $\psi' = \psi''$, that is, $\psi'$ lands at a co-root of $V$. 
\begin{proposition}\label{prop:wimpy}
Given a rational angle $\beta \in (\psi',\psi)$, the $\mu_{d+1}-$orbit of $\beta$ intersects the interior of $T^{stat}_{d+1,j}(\beta)$. 
\end{proposition}
\begin{proof}
We note that $\psi'_k = (d+1)^{k-1}\psi'$, $\beta_k = (d+1)^{k-1}\beta$ and $\psi_k = (d+1)^{k-1}\psi$ are in counterclockwise order. Suppose $|\beta - \psi'| = r \in \big(0,\frac{1}{(d+1)^k-1}\big)$, then
\begin{align*}
    |\beta_k - \psi'_k| &= (d+1)^{k-1}r\\
    \Big| \frac{\beta+j}{d+1} - \frac{\psi'+j}{d+1} \Big | & = \frac{r}{d+1}
\end{align*}
Since $\psi'_k = \frac{\psi'+j}{d+1}$, we note that 
 $\frac{\psi'+j}{d+1},\frac{\beta+j}{d+1}$ and $\beta_k$ are also in counterclockwise order, and 
\begin{align*}
    \Big| \beta_k - \frac{\beta+j}{d+1} \Big| & = |\beta_k - \psi'_k| - \Big| \frac{\beta+j}{d+1} - \frac{\psi'+j}{d+1} \Big | \\ & = (d+1)^{k-1}r - \frac{r}{d+1} \\& < \frac{1}{d+1}
\end{align*}
Arguing similarly, 
\begin{align*}
    \Big| \beta_k - \frac{\beta+j+1}{d+1} \Big| < \frac{1}{d+1}
\end{align*}
Since $d\geq 2$, this means that 
$$\beta_k \in  \Big[\frac{\beta+j}{d+1},\frac{\beta+j+1}{d+1}\Big]$$
$\psi',\psi$ are consecutive among angles periodic of period $k$ under $\mu_{d+1}$, hence $\beta $ cannot be periodic of period $k$. Thus,  \begin{align*}
    \beta_k \in \Biggl(\frac{\beta+j}{d+1},\frac{\beta+j+1}{d+1}\Biggl)
\end{align*}
\end{proof}
\begin{figure}[t]
\begin{subfigure}[b]{0.5\textwidth}
    \centering
    \includegraphics[width=\textwidth]{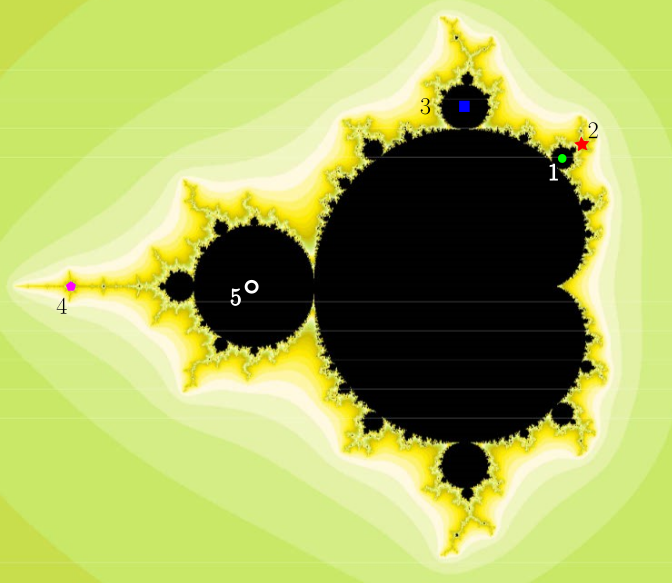}
    \caption{A few psf parameters in $\mathcal{P}_2$}
    \label{fig:L2_L3_embedding}
\end{subfigure}\\
  \begin{subfigure}[b]{0.4\textwidth}
        \centering
    \includegraphics[width=1.1\textwidth]{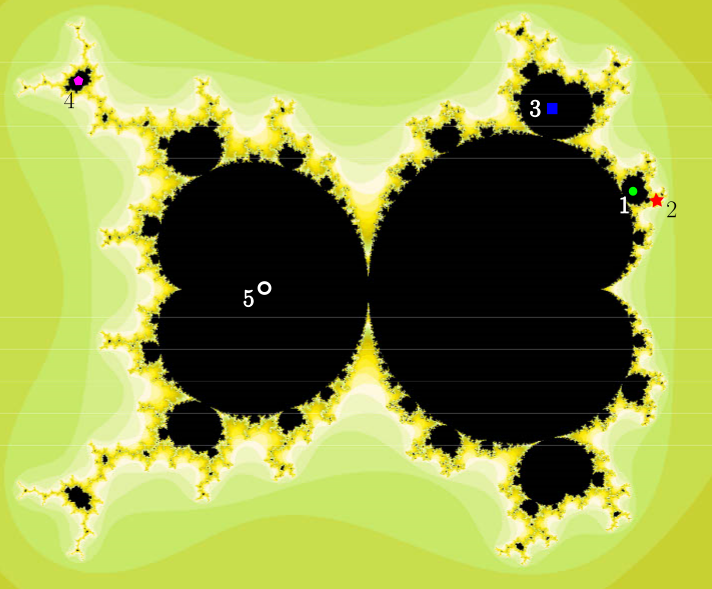}
    \caption{An illustration of $\mathcal{E}_0:\mathcal{P}_2 \longrightarrow \mathcal{P}_3 $}
    \label{fig:E_0_embedding}
  \end{subfigure}\hspace{30pt}
   \begin{subfigure}[b]{0.4\textwidth}
        \centering
    \includegraphics[width=1.1\textwidth]{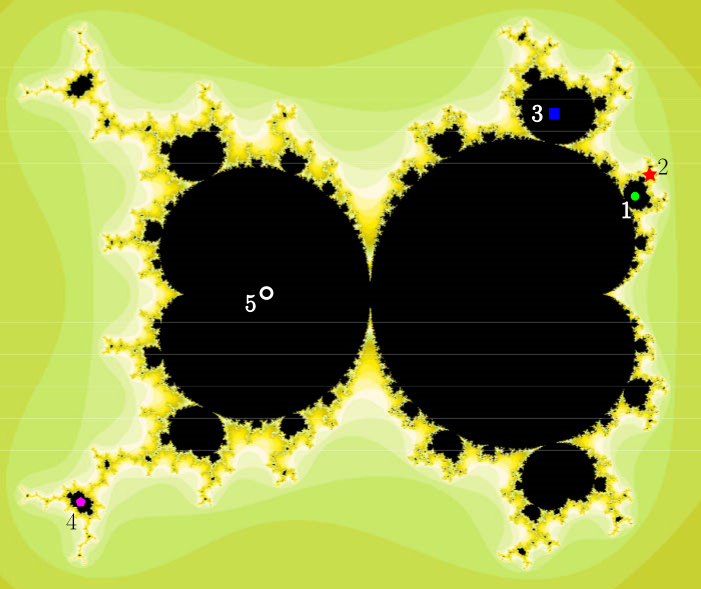}
    \caption{An illustration of $\mathcal{E}_1:\mathcal{P}_2 \longrightarrow \mathcal{P}_3 $}
    \label{fig:E_1_embedding}
  \end{subfigure}
  \caption{Images of a few parameters in $\mathcal{P}_2$ under $\mathcal{E}_0$ and $\mathcal{E}_1$. Parameters in the top picture are mapped to those in the bottom row pictures with matching number labels}
  \label{fig:E_0_1}
\end{figure}
For any $\delta \in \Q/\Z$, by Section~\ref{sec:proofoflemma1.1}, the $\mu_{d+1}$-orbit of $Z_j(\delta)$ does not intersect $T_{d+1,j}^{stat}(\delta)$. By the above proposition, we see that the image of $Z_j$ does not intersect $(\psi',\psi)$. 

\subsubsection{Critically preperiodic parameters}
Now assume that $\lambda$ has preperiod $\ell \geq 1$ and period $k \geq 1$. Choose an angular coordinate $\theta \in M_d(c)$, and let $c_{j,\theta}$ be the landing point of $\phi = Z_j(\theta)$.
\begin{proposition}\label{prop:preperiodic_number of accesses}$\Omega_{d+1}(c_{j,\theta}) = Z_j(\Omega_d(c))$; in particular, $c$ and $c_{j,\theta}$ have the same number of angular coordinates.
\end{proposition}
\begin{proof}
By Proposition~\ref{prop:branchpres}, $Z_j(\Omega_d(c)) \subset \Omega_{d+1}(c_{j,\theta})$. Given $\phi' \in \Omega_{d+1}(c_{j,\theta})$ with $\phi' \neq \phi$, the $\mu_{d+1}$-orbit of $\phi'$ does not intersect $T^{stat}_{d+1,j}(\phi')$. By Proposition~\ref{prop:maximalityprep}, there exists an angle $\theta'$ with $Z_j(\theta') = \phi'$. 

Let $\nu$ be the kneading sequence of $\theta$. Note that the itinerary of $\phi'$ with respect to $\phi$ is $\nu^j$. By construction, we can show that the itinerary of $\theta'$ with respect to $\theta$ is $\nu$. This implies that in the dynamical plane of $z^d+c$, the dynamic ray at angle $\theta'$ lands at $c$. Thus, $\theta'$ lands at $c $ in $\mathcal{M}_d$, implying $\phi' \in Z_j(\Omega_d(c))$. 
\end{proof}
\begin{example}
Fix $d=2$.
    The angles $\frac{17}{240}, \frac{19}{240}, \frac{23}{240}$ and $\frac{31}{240}$ all land at $c \approx 0.63571 +0.59124 i  \in \mathcal{M}_2$. Their corresponding images $\frac{163}{6840}$, $\frac{169}{6840}$, $\frac{187}{6840}$ and $\frac{241}{6840}$ under $Z_1$ land at the approximate value $ 0.62745 + 0.29882i \in \mathcal{M}_3$.
\end{example}
\begin{proof}[Proof of Lemma~\ref{lemma:combi_embeddings}] 
  Given $\lambda \in \mathcal{P}_d$, choose  $c \in M_d(\lambda)$. By Proposition~\ref{prop:spidermapsareconjugate}, $\lambda$ and $\mathcal{E}_j(\lambda)$ have the same dynamics on their postsingular sets, which shows (1) in Definition~\ref{defn:combi_embedding}. The properties (2) and (3) are clear by Propositions~\ref{prop:missedsector} 
~\ref{prop:addresspres} respectively.
  
  We show that property (4) is true:  given $\mu>\lambda$, there exist angles $\theta,\theta' \in \Theta_d(\lambda)$ that land at the same point in $\mathcal{M}_d$, and $\alpha \in \Theta_d(\mu)$ such that $\theta < \alpha < \theta'$. By monotonicity of $Z_j$, we have $Z_j(\theta) < Z_j(\alpha) < Z_j(\theta')$. By definition of  $\mathcal{E}_j$, for all $\lambda$, $Z_j(\Theta_d(\lambda)) \subseteq \Theta_{d+1}(\mathcal{E}_j(\lambda))$. Thus shows that $\mathcal{E}_j(\lambda) < \mathcal{E}_j(\mu)$. 

  Lastly, we show $\mathcal{E}_j$ is injective. Suppose $\mathcal{E}_j(\lambda) = \mathcal{E}_j(\lambda') = \mu$ for $\lambda \neq \lambda'$. Pick monic representatives $c,c'$ for $\lambda,\lambda'$ respectively that are in the sub-wake $(0,\frac{1}{d-1})$.  
  \begin{itemize}
      \item If $c$ (and therefore $c'$) are critically periodic, let $\theta,\theta'$ be the companion pair in $\Omega_d(c)$ and $\alpha,\alpha'$ be the companion pair in $\Omega_d(c')$. Without loss of generality, we have 
      \begin{align*}
          0 < \theta < \theta' < \alpha < \alpha' < \frac{1}{d-1}
      \end{align*}
      But this implies
            \begin{align*}
          0 < Z_j(\theta) < Z_j(\theta') < Z_j(\alpha) < Z_j(\alpha' )< \frac{1}{d}
      \end{align*}
      By Proposition~\ref{prop:companion}, $( Z_j(\theta) , Z_j(\theta'))$, and $(Z_j(\alpha), Z_j(\alpha' ))$ are companion pairs. The above inequality implies that they land on different hyperbolic components, but since both components are in the sub-wake $(0,\frac{1}{d})$, the centers of these components cannot both be monic representatives for $\mu$. This presents a contradiction.
      \item If $c$ (and therefore $c'$) are critically preperiodic, choose $\theta,\theta' \in \Omega_d(c)$ and $\alpha,\alpha' \in \Omega_d(c')$. Again without loss of generality, we have 
        \begin{align*}
          0 < \theta < \theta' < \alpha < \alpha' < \frac{1}{d-1}
      \end{align*}
          \begin{align*}
          0 < Z_j(\theta) < Z_j(\theta') < Z_j(\alpha) < Z_j(\alpha' )< \frac{1}{d}
      \end{align*}
      Let $x$ be the landing point of $Z_j(\theta)$.  Proposition~\ref{prop:preperiodic_number of accesses}, $\Omega_{d+1}(x) = Z_j(\Omega_d(c))$, and so $Z_j(\alpha)$ and $Z_j(\alpha' )$ land at $y \neq x$. Since $x$ and $y$ are both in the sub-wake $(0,\frac{1}{d})$, they are not both monic representatives of $\mu$, which is a contradiction.  
  \end{itemize}
\end{proof}
Figure~\ref{fig:E_0_1} illustrates $\mathcal{E}_0,\mathcal{E}_1: \mathcal{P}_2 \longrightarrow \mathcal{P}_3$ on a few input points.

\section{Proof of Lemma~\ref{lemma: external address correspondence}}\label{sec:proofoflemma1.3}
Given $\ell\geq 0,k\geq 1$, and a polynomial $q \in \Z[x]$, we define the function $\tau_q: \Z_{\geq 2} \longrightarrow \Q/\Z$ by $\tau_q(x)= \frac{(x-1)q(x)}{x^\ell(x^k-1)} (\text{mod }1)$. 

\begin{proposition}\label{prop:poly_translated}
Given $\ell\geq 0,k\geq 1$, and $\tau_q$ defined as above, 
    $\tau_q \equiv \tau_{q'}$ if and only if $q'-q$ is a multiple of $x^\ell(x^k-1 + x^{k-2}+....+1)$. 
\end{proposition}
\begin{proof}
    \begin{align*}
        \tau_q &\equiv \tau_{q'} \\
        \iff \frac{(x-1)(q(x) - q'(x))}{x^\ell(x^k-1)} &\equiv 0\\
        \iff x^\ell(x^k-1) &| (x-1)(q(x) - q'(x))\\ 
        \iff x^\ell(x^{k-1}+x^{k-2}+...+1) &|(q(x) - q'(x))
    \end{align*}
    
\end{proof}

\begin{proposition}\label{prop:address to poly}
    Given an external address $\underline{s} = s_1s_2.....$, with $s_1 = 0$, preperiod $\ell\geq 0$ and period $k\geq 1$, for $d\geq \max |s_n|$, let $\theta_d$ be the angle with $d$-adic expansion $.x_1(d)x_2(d)...x_\ell(d) \overline{x_{\ell+1}(d)x_{\ell+2}(d)....x_{\ell+k}(d)}$ given by
    \begin{align*}
        x_n(d) & = \begin{cases}
            s_n & s_n\geq 0\\
            d-|s_n| & s_n <0
        \end{cases}
    \end{align*}
    
    \noindent Then there exists a degree $D\geq 2$, $j \in \{0,1,....,D-1\}$, $q \in \Z[x]$ with $\deg q \leq \ell+k-2$ such that 
    \begin{align*}
        \theta_{d+1} &=Z_j(\theta_d)\\
        (d-1)\theta_d &= \tau_q(d)
    \end{align*}
    for all $d\geq D$.
\end{proposition}
\begin{proof}
    
Let $D = 1+2(\max |s_n|+1)$. 
For all $n$, and $d\geq D$,
\begin{align*}
    0 &\leq x_n(d) \leq \frac{D-3}{2}  \hspace{10 pt} \text{ if } s_n\geq 0\\
     \frac{D+3}{2}& \leq x_n(d) \leq d-1  \hspace{10 pt}\text{ if } s_n < 0
     \end{align*}
The interval $(\frac{D-3}{2},\frac{D+3}{2})$ contains at least one integer in $\{0,1,...,D-1\}$, and for any integer $j$ in this interval, for all $d\geq D$, we have
\begin{align*}
    \theta_{d+1} & = Z_j(\theta_d)
\end{align*}
Let $\widetilde{q}(x) = \sum_{n=1}^\ell x_n(x) x^{\ell-n}(x^k-1) + \sum_{n=1}^k x_{\ell+n}(x)x^{k-n}$. 
We also have 
\begin{align*}
    \theta_d & = \frac{\widetilde{q}(d)}{d^\ell(d^k-1)}
\end{align*}

Since $x_1(d) = s_1 = 0$  and $x_2(d)$ is at most linear, $\deg \Tilde{q} \leq \ell+k-1$. If $\deg \Tilde{q} \leq \ell+k-2$, set $q = \Tilde{q}$. Otherwise, note that $\Tilde{q}$ has degree $\ell+k-1$ if and only if $x_2(d)$ is of the form $d-|s_2|$. In this case, the leading coefficient of $\Tilde{q}$ is $1$. Set $q = \Tilde{q}-x^{\ell}(x^{k-1}+x^{k-2}+...+x+1)$. $q$ has degree $\leq \ell+k-2$, and we have 
\begin{align*}
   (d-1) \theta_d & = \tau_q(d)
\end{align*}
 by Proposition~\ref{prop:poly_translated}.\\

\end{proof}
\begin{example}
Given $r \in \Z$, for $\underline{s} = 0\overline{r}$, we have two separate cases:
\begin{itemize}
    \item If $r\geq 0$, then $D = 2r+3$,  $\theta_d = \frac{r}{d(d-1)}$ for all $d\geq D$, and  $q(x) = r$.
    \item If $r<0$, then $D = 2r+3$, $\theta_d = \frac{d-|r|}{d(d-1)}$ for all $d\geq D$, and $q(x) = r$.
\end{itemize}
\end{example}

\begin{example}\label{ex:last example}
    Let $\underline{s} = 000(-1)\overline{0010}$. 
We have $D = 5$, and for all $d\geq D$,
\begin{align*}
    \theta_d& = .000(d-1)\overline{0010} =  \frac{d^5 - d^4 +1}{d^4(d^4-1)} 
\end{align*}
and we have $q(x) = x^4 +x - 1$.

In fact, letting $\theta = \frac{17}{2^4(2^4-1)} = \frac{17}{2^4(2^4-1)}$, we have  $\theta_d = Z_1^{\circ (d-2)}(\theta)$ for all $d\geq D$.
\end{example}
\begin{figure}[t]
    \centering
    \includegraphics[scale=0.5]{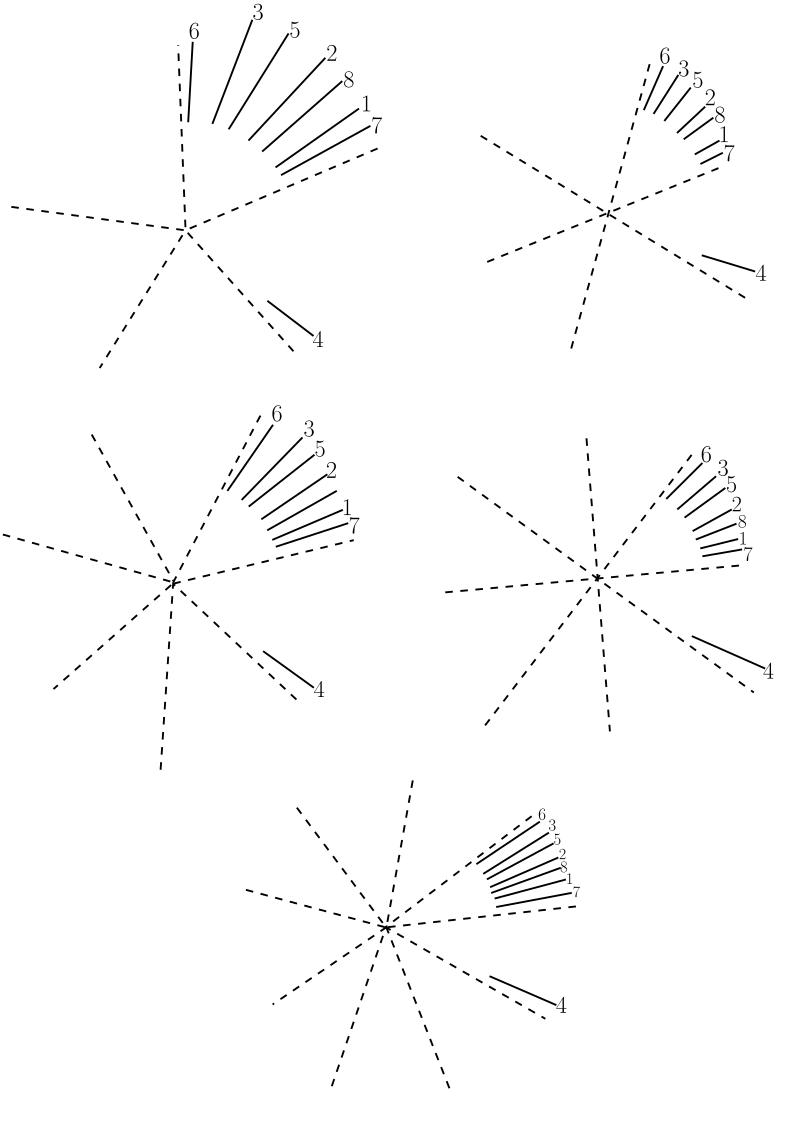}
    \caption{The spiders $S_d(\theta _d)$ for $d=5,6,7,8,9$, when $\underline{s}$ is set to $0001\overline{0010}$ (the dotted lines are not part of the spider, but indicate the additional legs in $\widetilde{S}_d(\theta_d)$). The leg labelled $n$ corresponds to $d^{n-1}\theta$. Compare these with Figure~\ref{fig:exp_spider} }
    \label{fig:sequence_spider}
\end{figure}
\begin{definition}
    Given $\theta \in \Q/\Z$, degree $d\geq 2$ and a preperiodic address $\underline{s}$ with preperiod $\ell\geq 1$ and $k\geq 1$, we say that $S_D(\theta_D)$ and $\Gamma_{\underline{s}}$ are isomorphic if
    \begin{enumerate}
        \item They have the same number of legs
        \item They have the same circular order of legs, in the following sense: 

        If $(\theta_{n_1},\theta_{n_2},....,\theta_{n_{\ell+k}})$ are in counterclockwise order in $\mathscr{O}_d(\theta)$, then $$\sigma^{\circ (n_1 - 1)}(\underline{s}) << \sigma^{\circ (n_2 - 1)}(\underline{s}) << ... <<\sigma^{\circ (n_{\ell+k} - 1)}(\underline{s})$$
    \end{enumerate}
\end{definition}
\begin{proposition}\label{same_spider_exp}
Let $\underline{s}$ be an external address with preperiod $\ell\geq 1$ and period $k\geq 1$, and let  $D$, $\theta_D$ be as defined in  Proposition~\ref{prop:address to poly}. 
    The spiders $S_D(\theta_D)$ and $\Gamma_{\underline{s}}$ are isomorphic.
\end{proposition}
\begin{proof}

    Let $\theta = \theta_D$ and $\theta_n = D^{n-1}\theta_D$ for all $n\geq 1$. 

    Suppose $D^{n-1}\theta < D^{m-1}\theta$, then
    \begin{align*}
        x_n(D)x_{n+1}(D)..... < x_m(D)x_{m+1}(D)....
    \end{align*}
    Let $r\geq 0$ be the least integer such that $x_{n+r}(D) \neq x_{m+r}(D)$. Then for all $r'<r$, there are two possibilities:
    \begin{itemize}
        \item If $0 \leq x_{n+r'}(D) \leq \frac{D-3}{2} $, then $s_{n+r'} = x_{n+r'}(D) = x_{m+r'}(D) = 
        s_{m+r'}$.
        \item If $\frac{D+3}{2} \leq x_{n+r'}(D) \leq D-1$, then $s_{n+r'} = x_{n+r'}(D) - D = x_{m+r'}(D) -D = s_{m+r'}$
    \end{itemize}
    That is, for all $r'<r$, $s_{n+r'} = s_{m+r'}$. 

    At the index $n+r$ we have $x_{n+r}(D) <  x_{m+r}(D)$. There are only three possibilities:

    \begin{itemize}
        \item If $x_{m+r}(D) \leq \frac{D-3}{2}$, then  $s_{n+r} = x_{n+r}(D) < x_{m+r}(D) = s_{m+r} $
        \item If $x_{n+r}(D) \geq \frac{D+3}{2}$, then  $s_{n+r} = x_{n+r}(D)-D < x_{m+r}(D)-D = s_{m+r} $
        \item If $x_{n+r}(D) \leq \frac{D-3}{2}$ and $x_{m+r}(D) \geq \frac{D+3}{2}$, then  $s_{n+r}>0$ and $s_{m+r}(D) < 0$
    \end{itemize}
    
In the first two cases, we directly get $\sigma^{\circ (n-1)}(\underline{s}) < \sigma^{\circ (m-1)}(\underline{s})$. In the third case, $\sigma^{\circ (n-1)}(\underline{s})>\overline{0}$ and $\sigma^{\circ (m-1)}(\underline{s}) < \overline{0}$. Thus $\sigma^{\circ (n-1)}(\underline{s})<<\sigma^{\circ (m-1)}(\underline{s})$.

Lastly, if $D^{n-1}\theta = \frac{\theta+r}{D}$, then $x_n(D) = r$ and $x_m(D) = 0$ for all $m>n$. The latter condition implies $s_m = 0$ for all $m>n$ 
\begin{itemize}
    \item If $0\leq r\leq \frac{D-3}{2}$, then  $\sigma^{\circ (n-1)}(\underline{s}) = r\overline{0}$.
    \item If $\frac{D-3}{2}\leq r\leq D-1$, then $\sigma^{\circ (n-1)}(\underline{s}) = -c\overline{0}$
\end{itemize}
This shows that the spiders $S_D(\theta_D)$ and $\Gamma_{\underline{s}}$ are isomorphic.
\end{proof}
\begin{corollary}
The dynamical systems $\mathcal{F}_{D,\theta_D}: S_D(\theta_D)\righttoleftarrow $ and $\mathcal{F}_{\underline{s}}: \Gamma_{\underline{s}} \righttoleftarrow$ are conjugate by a map $h$ that satisfies
\begin{enumerate}
    \item $h(\infty) = \infty$
    \item $h(e^{2\pi i \theta}) = \tilde{e}_1$
    \item $h$ preserves circular order of legs
\end{enumerate}

\end{corollary}
\begin{proof}
For $r \in \{0,1,...,D-1\}$, with $r+c = D$,  want to show that $T_{d,r}^{stat}(\theta_D)$ corresponds to $T_r(\underline{s})$
if $r+1 \leq \frac{D-3}{2}$, and to $T_{-(c+1)}(\underline{s})$ if $r+1 \geq \frac{D+3}{2}$. 

    Suppose $D^{n-1}\theta \in T_{d,r}^{stat}(\theta_D)$, then $x_n(D) \in \{r,r+1\}$. 
\begin{itemize}
 \item If $r+1 \leq \frac{D-3}{2}$, then $0 \leq x_n(D) \leq \frac{D-3}{2}$, and thus $s_n \in \{r,r+1\}$, implying $\sigma^{\circ (n-1)}(\underline{s}) \in T_r(\underline{s})$. 
 \item If $r+1 \geq \frac{D+3}{2}$, since $x_n(D) \not \in (\frac{D-3}{2},\frac{D+3}{2})$ we must have $r\geq \frac{D+3}{2}$, and so $s_n = D-x_n(D) \in \{-c,-(c+1)\}$, implying $\sigma^{\circ (n-1)}(\underline{s}) \in T_{-(c+1)}(\underline{s})$.
\end{itemize}

    Combining this with Proposition~\ref{same_spider_exp}, the   statement is clear.

\end{proof}
\begin{remark}
    $[\mathcal{F}_{D,\theta_D}]: [S_D(\theta_D)\righttoleftarrow] $ and $[\mathcal{F}_{\underline{s}}]: [\Gamma_{\underline{s}}] \righttoleftarrow$ are conjugate by a map $h$ that preserves circular order and maps $\infty$ to $\infty$, and $x_1$ to $e_1$.
\end{remark}
Propositions~\ref{prop:address to poly}, \ref{same_spider_exp} and \ref{prop:spiderconjugacy} together give a proof of Lemma~\ref{lemma: external address correspondence}. See Figure~\ref{fig:sequence_spider} for an illustration of $S_d(\theta_d)$ for $\underline{s} = 0001\overline{0010}$. 

\section{Proof of Lemma~\ref{lemma: convergence of spider maps}}
\label{sec:proofoflemma1.4}
Let $\underline{s}$ be an external address with preperiod $\ell\geq 1$ and period $k\geq 1$, and let $D$, $\theta_d$ be as defined in  Proposition~\ref{prop:address to poly}. By Proposition~\ref{same_spider_exp} and Proposition~\ref{prop:spidermapsareconjugate}, for each $d\geq D$, there exists a homeomorphism $h_d:[S_d(\theta)] \longrightarrow [\Gamma_{\underline{s}}]$  preserving circular order of legs that conjugates $[\mathcal{F}_{d,\theta_d}]$ to $[\mathcal{F}_{\underline{s}}]$, that maps $\infty$ to $\infty$ and $x_1$ to $e_1$. Using the Alexander trick, we can extend $h_d$ to a homeomorphism of the sphere. This gives a way of extending $[\mathcal{F}_{\underline{s}}]$ to a degree $d-$ map on $S^2$, given by $h_d \circ [\mathcal{F}_{d,\theta_d}] \circ h_d^{-1}$. By this definition this map is Thurston equivalent to $[\mathcal{F}_{d,\theta_d}]$. 

    Define an operator $\sigma^d_{\underline{s}}$ on $\mathcal{T}_{\underline{s}}$ as follows: given $[\phi] \in \mathcal{T}_{\underline{s}}$ with $\phi(e_1) = 0$, define $\sigma^d_{\underline{s}}([\phi]) = [\psi]$, where $\psi$ completes the following diagram:
    \[
\begin{tikzcd}
\Big([\Gamma_{\underline{s}}],A_{\underline{s}}\Big) \arrow[r,"\psi"] \arrow[d,"{[\mathcal{F}_{\underline{s}}]}"]&\Big(\C,\psi(A_{\underline{s}})\Big) \arrow[d,"f_\phi"]\\
\Big([S_{\underline{s}}],A_{\underline{s}}\Big) \arrow[r,"\phi"] & \Big(\C,\phi(A_{\underline{s}})\Big)
\end{tikzcd}
\]
where $f_\phi(z) = \phi(e_2)(1+\frac{z}{d})^d$.
Define the map $H_d: \mathcal{T}^d_{\theta_d} \longrightarrow  \mathcal{T}_{\underline{s}}$ as $H_d([\phi]) = [\phi \circ h_d^{-1}]$ (it is easy to see that this is well-defined).
\begin{proposition}\label{prop:Hd is isometry}
    $H_d$ is a biholomorphism that is an isometry with respect to the Teichm\"{u}ller metrics on $\mathcal{T}^d_{\theta_d}$ and $\mathcal{T}_{\underline{s}}$.   
\end{proposition}
\begin{proof}
    It is clear that $H_d$ is a biholomorphism. For all $[\phi],[\psi] \in \mathcal{T}^d_{\theta_d}$,
    \begin{align*}
        d_T(H_d([\phi]), H_d([\psi])) & = \inf_{q.c. \psi' \sim( \phi \circ h_d^{-1}) \circ (h_d \circ \psi^{-1})} K(\psi')\\& = \inf_{q.c. \psi' \sim \phi \circ \psi^{-1}} K(\psi') \\& = d_T([\phi],[\psi])
    \end{align*}
\end{proof}
\begin{proposition}\label{prop: conjugate with poly spider operator}
    $\sigma^d_{\underline{s}}$ is equivalent to $\sigma^d_{\theta_d}$ and weakly contracting with respect to the Teichm\"{u}ller metric on $\mathcal{T}_{\underline{s}}$. Let $[\phi_d]$ be the fixed point of $\sigma^d_{\theta_d}: \mathcal{T}^d_{\theta_d} \righttoleftarrow$. Then $[ \phi_d \circ h_d^{-1} ]$ is the fixed point of $\sigma^d_{\underline{s}}$. 
\end{proposition}
\begin{proof}
  By definition, we have $\sigma^d_{\underline{s}} = H_d \circ \sigma^d_{\theta_d} \circ H_d^{-1}$. Combined with Proposition~\ref{prop:Hd is isometry}, the statement follows. 
\end{proof}

\begin{proposition} \label{prop: qc homeos fixing three points}
    Let $\varphi_n$ be a sequence of $K_n$-quasiconformal homeomorphisms of $\widehat{\C}$ fixing three distinct points $a, b, c$ of $\widehat{\C}$. Suppose that $K_n \rightarrow 1$ as $n$ tends to $\infty$, then $\varphi_n\longrightarrow \mathrm{id}_{\C}$ in the spherical metric of $\hat{\C}$.
\end{proposition}

\begin{proof}
    Let $\varphi_{n_k}$ be a subsequence of the sequence $\varphi_n$. By \cite[Theorem 1.26]{branner_fagella_2014}, there exists a subsequence of $\varphi_{n_k}$ that converges  in the spherical metric to a quasiconformal map $\varphi \colon \widehat{\C} \to \widehat{\C}$ with $K(\phi) = 1$. By Weyl's Lemma (\cite[Theorem 1.14]{branner_fagella_2014}), $\varphi$ is holomorphic. Since $\varphi $ fixes three 
 distinct points of $\widehat{\C}$, it equals  $\id_{\widehat{\C}}$. 
 
 This argument applies to every subsequence of $\varphi_n$, thus the entire sequence $\varphi_n$ converges to $\id_{\widehat{\C}}$ in the spherical metric.
\end{proof}
\begin{remark}
    In particular, the $\phi_n$'s above converge to $\id_{\C}$ uniformly on compact subsets of $\C$.
\end{remark}

\begin{proposition}\label{pullback_convergence}Let $[\phi_d],[\phi] \in \mathcal{T}_{\underline{s}}$. 

    If  $\phi_d \longrightarrow \phi$ uniformly on compact sets of $[\Gamma_{\underline{s}}] \setminus \{\infty\}$, with $\phi_d(e_1) = \phi(e_1) = 0$ and $\phi_d(e_2) = \phi(e_2) = \lambda$, and $\psi_d \in [\sigma^d_{\underline{s}}([\phi_d])]$, $\psi \in [\sigma_{\underline{s}}(\phi)]$ are the unique maps  such that $p_{d, \lambda} \circ \psi_d = \phi_d \circ [\mathcal{F}_{\underline{s}}]$,  $p_\lambda \circ \psi =  \phi \circ [\mathcal{F}_{\underline{s}}]$, then $\psi_d \longrightarrow \psi$ uniformly on compact sets of $[\Gamma_{\underline{s}}] \setminus \{\infty\}$.  
\end{proposition}
\begin{proof}
Given $x \in [\Gamma_{\underline{s}}]$ on the edge $\gamma_{n}$ and $y = [\mathcal{F}_{\underline{s}}](x)$,  $\psi_d(x)$ is a point $w_d$ such that 
\begin{align*}
   z_d:= \phi_d(y) & = \lambda\Big(1+\frac{w_d}{d}\Big)^d
\end{align*}
Let $\gamma_1$ be the leg of $\phi$ with endpoint $\phi(e_1)$. We call a component of $\C \setminus p_{d,\lambda}^{-1}(\phi(\gamma_1))$ a sector.  Exactly one of these sectors contains the point $0$, and we label this sector as $x_1(d)$, while labelling the rest counterclockwise from $1$ to $d-1$. 

The point $w_d$ above is chosen so that it belongs in the sector labelled $x_n(d)$. Assuming that $z_d/\lambda=r_de^{it_d}$, where $t_d \in [0,\pi)$, we have $w_d = d\big(r_d^{\frac{1}{d}}e^{\frac{it_d}{d} + \frac{2\pi i x_n(d)}{d}} - 1\big)$. By the construction in Proposition~\ref{prop:address to poly}, $x_n(d)$ is equal to either $s_n$ or $d+s_n$. Therefore, $e^{\frac{2 \pi i x_n(d)}{d}} = e^{\frac{2 \pi i s_n}{d}}$ for all $d\geq D$, and we have 
\begin{align*}
    w_d & =d\Big(r_d^{\frac{1}{d}}e^{i\frac{t_d}{d}} \cos \frac{2\pi  s_n}{d} - 1\Big) + id\Big( r_d^{\frac{1}{d}}e^{i\frac{t_d}{d}} \sin \frac{2\pi  s_n}{d} \Big)
\end{align*}
 $r_d \longrightarrow r$ and $t_d \longrightarrow t$, where $re^{it} = \frac{\phi(y)}{\lambda}$. 

Suppose $re^{it} = \frac{\phi(y)}{\lambda}$ (where $t \in [0,\pi]$). Since $\phi_d \longrightarrow \phi$, we have 
\begin{align*}
    \lim_{d \rightarrow \infty}r_d & = r \\
    \lim_{d \rightarrow \infty}t_d& = t\\ 
    \lim_{d \rightarrow \infty}\psi_d(x) &= \ln r + it + 2\pi i s_n= \psi(x)
\end{align*}
This last convergence is uniform on compact subsets of  $[\Gamma_{\underline{s}}] \setminus \{\infty\}$, and proves the proposition.    
\end{proof}
\begin{proposition}\label{conv_implies_teich_convergence}Given $[\phi_n], [\phi] \in \mathcal{T}_{\underline{s}}$, if $\phi_n \longrightarrow \phi$ uniformly on compact subsets of $[\Gamma_{\underline{s}}] \setminus \{\infty\}$ and we have $\phi_n(e_1) = \phi(e_1) = 0$ for all $n$, then $[\phi_n] \longrightarrow [\phi]$ in $\mathcal{T}_{\underline{s}}$. 
\end{proposition}
\begin{proof}
Without loss of generality, we may also assume $\phi_n(e_2) = \phi(e_2) = \lambda \in \C^*$ for all $n$. Since $\phi_n \longrightarrow \phi$, we have $\phi_n(e_j) \longrightarrow \phi(e_j)$ for all $j$. Thus in the moduli space $\Mod(\underline{s})$, $[[\phi_n]] \longrightarrow [[\phi]]$. 

By Proposition~\ref{prop:fiber over moduli space}, there exists a sequence $\langle h_n\rangle \in \PMCG(S^2, \phi(A_{\underline{s}}) \cup \{\infty\})$ such that 
\begin{align*}
    d_T([\phi_n], [h_n \circ \phi]) \rightarrow 0
\end{align*}

By definition of Teichm\"{u}ller distance, there exists a sequence of quasiconformal maps $\psi_n $ isotopic to $\phi_n \circ \phi^{-1} \circ h_n^{-1}$ such that $K(\psi_n) \rightarrow 1$. 
We note that $\psi_n(\infty) = \infty, \psi_n(0) = 0$ and $\psi_n(\lambda) = \lambda$. By Proposition~\ref{prop: qc homeos fixing three points}, $\psi_n \longrightarrow \mathrm{id}_{\C}$ uniformly on compact subsets of $\C$. By construction, for each $n$, $\langle h_n \rangle  = \langle \phi_n^{-1} \circ \phi \circ  \psi_n^{-1}\rangle \in \PMCG(S^2,A_{\underline{s}} \cup \{\infty\})$. 

Since $\phi_n^{-1} \circ \phi \circ \psi_n^{-1} \longrightarrow \textrm{id}_{\C}$, we have $\langle h_n\rangle  \longrightarrow \langle \textrm{id}_{S^2}\rangle $ in $\PMCG(S^2, A_{\underline{s}} \cup \{\infty\})$. Since $\PMCG(A_{\underline{s}} \cup \{\infty\})$ is a discrete group, we must have $\langle h_n\rangle  = \langle \textrm{id}\rangle $ eventually, and thus, 
\begin{align*}
    d_T([\phi_n], [\phi]) \rightarrow 0
\end{align*} 
\end{proof}
\begin{proposition}
$\sigma^d_{\underline{s}} \longrightarrow \sigma_{\underline{s}}$ pointwise on $\mathcal{T}_{\underline{s}}$.
\end{proposition}
\begin{proof}
For a given $d \geq D$ and $[\phi] \in \mathcal{T}_{\underline{s}}$ with $\phi(e_1) = 0$, let $\lambda = \phi(e_2)$ and $\psi_d \in \sigma_d([\phi])$, $\psi \in \sigma_{\underline{s}}([\phi])$ be such that $p_{d,\lambda} \circ \psi_d = \phi \circ [\mathcal{F}_{\underline{s}}]$ and $p_\lambda \circ \psi = \phi \circ [\mathcal{F}_{\underline{s}}]$. By Proposition~\ref{pullback_convergence}, $\psi_d \longrightarrow \psi$ locally uniformly on 
 $[\Gamma_{\underline{s}}] \setminus \{\infty\}$. By Proposition~\ref{conv_implies_teich_convergence}, $[\psi_d] \longrightarrow [\psi]$.    
\end{proof}
\begin{proposition}\label{prop:local uniform spider convergence}
$\sigma^d_{\underline{s}} \longrightarrow \sigma_{\underline{s}}$ uniformly on compact sets of  $\mathcal{T}_{\underline{s}}$. 
\end{proposition}
\begin{proof}
Let $U$ be a compact set in $\mathcal{T}_{\underline{s}}$. Given any $\epsilon>0$,  let $\{B_1,....,B_r\}$ be a covering of $\overline{U}$ by $\frac{\epsilon}{3}-$balls, and let $[\varphi _i]$ be the center of $B_i$. 

There exists an integer $D$ such that for all $d\geq D$, and all $i$, 
\begin{align*}
    d_{T}(\sigma^d_{\underline{s}}([\varphi_i]), \sigma_{\underline{s}}([\varphi_i])) < \frac{\epsilon}{3}
\end{align*}
Since $\sigma_d$ is weakly contracting, for each $[\varphi] \in U$, there exists $i$ such that 
\begin{align*}
d_{T}(\sigma_{\underline{s}}([\varphi_i]),\sigma_{\underline{s}}([\varphi])) < \frac{\epsilon}{3}\\
    d_T(\sigma^d_{\underline{s}}([\varphi]),\sigma^d_{\underline{s}}([\varphi_i])) < \frac{\epsilon}{3}
\end{align*}
Using the three equations above, we have
\begin{align*}
    d_{T}(\sigma^d_{\underline{s}}([\varphi]),\sigma_{\underline{s}}([\varphi])) < \epsilon
\end{align*}    
\end{proof}
Lemma~\ref{lemma: convergence of spider maps} follows by Propositions~\ref{prop:address to poly}, ~\ref{prop: conjugate with poly spider operator} and ~\ref{prop:local uniform spider convergence}. 
\section{Proof of Theorem~\ref{thm:A}}\label{sec:maintheorem}
Given $\lambda \in \mathcal{P}$, choose a preperiodic external address $\underline{s} \in \Theta(\lambda)$, with preperiod $\ell\geq 1$ and period $k\geq 0$. From Section~\ref{sec:spider algo for exp}, we know that $\sigma_{\underline{s}}$ has a fixed point $[\phi] \in \mathcal{T}_{\underline
{s}}$ such that $\phi(e_1) = 0$, and there exists $\psi \in [\phi]$ such that $\phi \circ [\mathcal{F}_{\underline{s}}] \circ \psi^{-1} = \lambda \exp(z)$. 

Let $\sigma^d_{\underline{s}}$ be as defined in Section~\ref{sec:proofoflemma1.4}, and let $[\phi_d]$ be the unique fixed point of $\sigma^d_{\underline{s}}$ in $\mathcal{T}_{\underline{s}}$. 
\begin{proposition}
As $d \rightarrow \infty$,  $[\varphi_d]\longrightarrow [\varphi]$. 
\end{proposition}
\begin{proof}
Since $\mathcal{T}_{\underline{s}}$ is a manifold, it is locally compact. It is also geodesically complete with respect to the Teichm\"{u}ller metric (see \cite[Section 11]{primer}).  By the Hopf-Rinow theorem \cite[page 8]{gromov} , any closed and bounded subset of $\mathcal{T}_{\underline{s}}$ is compact. 
\begin{figure}
    \centering
    \includegraphics[scale=0.3]{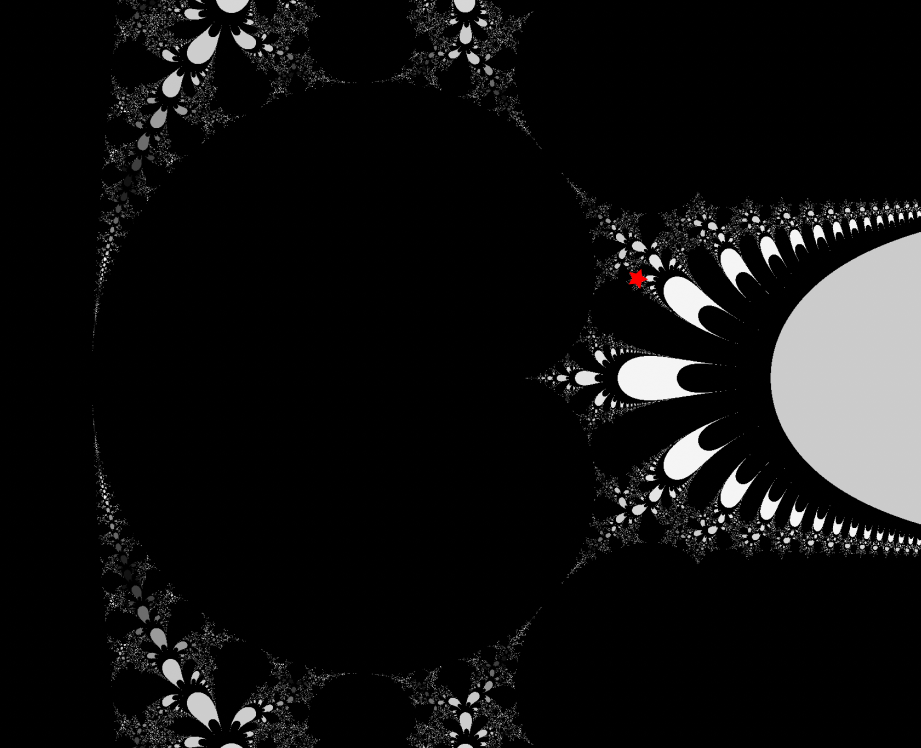}
    \caption{The star marks the position of $\lambda \approx 1.16302+ 0.71056 i$, the landing point of the parameter ray at address $0001\overline{0010}$ in the exponential parameter plane}
    \label{fig:lambda}
\end{figure}
Given $\epsilon >0$, let $B_\epsilon$ be a ball of radius $\epsilon$ centered at $[\phi]$. By the above, $\overline{B}_\epsilon$ is compact, and by Proposition~\ref{prop: strong contraction}, there exists a constant $c_\epsilon \in [0,1)$ such that 
\begin{align*}
    d_T(\sigma_{\underline{s}}([\psi]), \sigma_{\underline{s}}([\phi'])) \leq c_\epsilon d_T([\psi],[\phi'])
\end{align*}
for all $[\psi],[\phi'] \in B_\epsilon$. 
By Proposition~\ref{prop:local uniform spider convergence}, there exists an integer $N$ such that for all $d \geq N$, 
\begin{align*}
    d_T(\sigma^d_{\underline{s}}([\psi]), \sigma_{\underline{s}}([\psi])) < (1-c_\epsilon)\epsilon
\end{align*}
for all $[\psi] \in B_\epsilon$.

Therefore, for all $[\psi] \in B_\epsilon$ and $d\geq N$, 
\begin{align*}
    d_T(\sigma^d_{\underline{s}}([\psi]), [\phi]) &\leq  d_T(\sigma^d_{\underline{s}}([\psi]), \sigma_{\underline{s}}([\psi])) + d_T(\sigma_{\underline{s}}([\psi]), \sigma_{\underline{s}}([\phi])) \\& < (1-c_\epsilon)\epsilon + c_\epsilon \epsilon  < \epsilon
\end{align*}
\begin{figure}[t]
   \begin{subfigure}[b]{0.45\textwidth}
       \centering
       \includegraphics[scale=0.37]{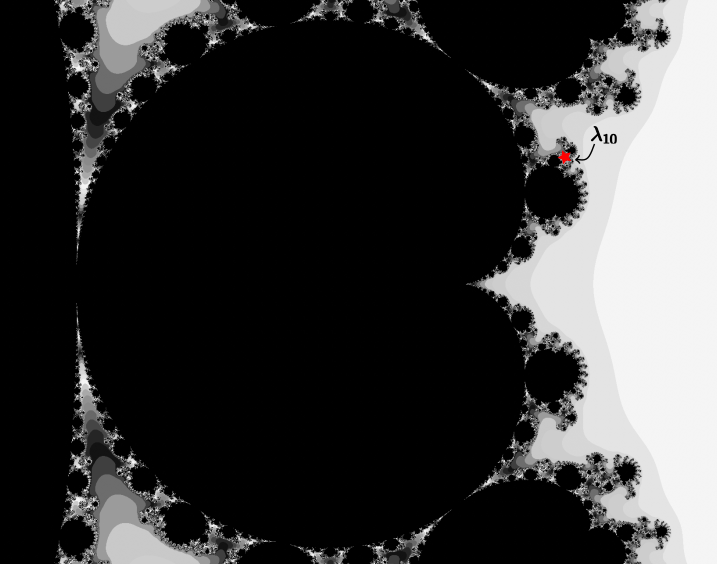}
       \caption{$\lambda_{10}\approx 1.1176 + 0.86608i$}
   \end{subfigure}\hspace{5pt}
   \begin{subfigure}[b]{0.45\textwidth}
       \centering
       \includegraphics[scale=0.31]{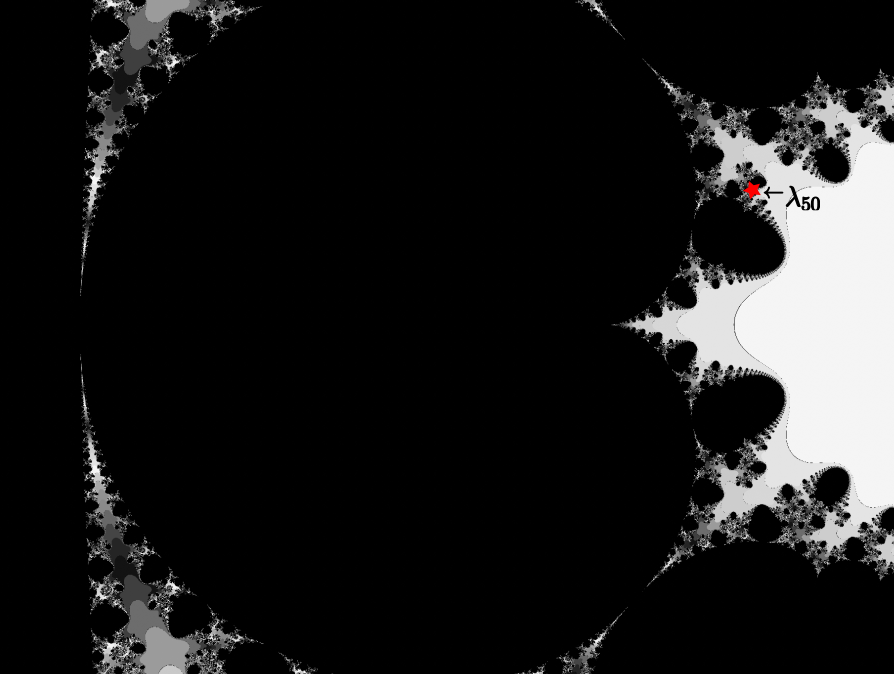}
         \caption{$\lambda_{50}\approx 1.1545 + 0.74281i$}
   \end{subfigure}\\ 
   \vspace{10pt}
    \begin{subfigure}[b]{0.45\textwidth}
       \centering
       \includegraphics[scale=0.32]{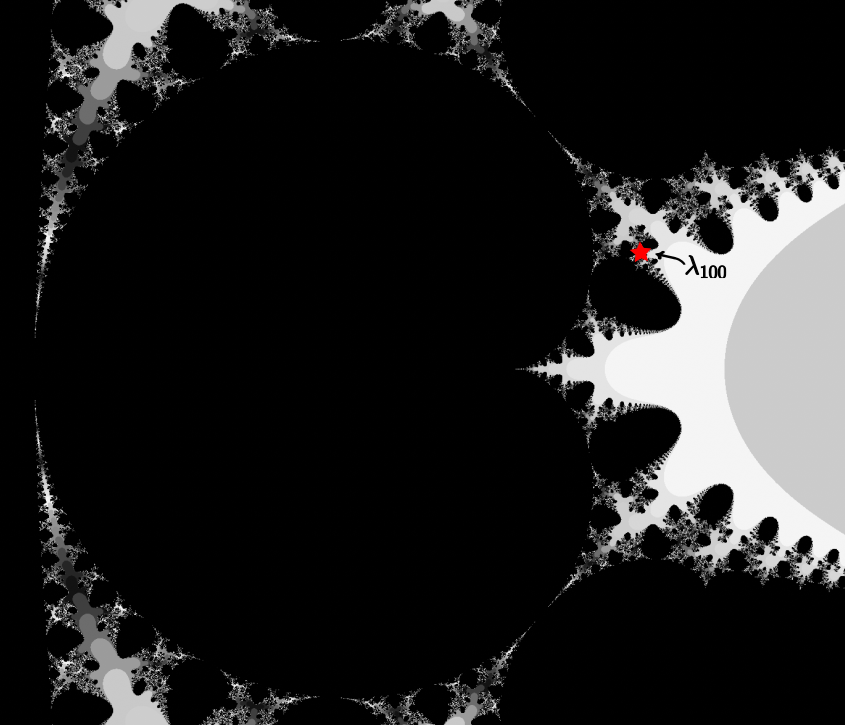}
       \caption{$\lambda_{100}\approx 1.1575 + 0.72671i$}
   \end{subfigure}\hspace{5pt}
     \begin{subfigure}[b]{0.45\textwidth}
       \centering
       \includegraphics[scale=0.32]{lambda100.png}
       \caption{$\lambda_{200}\approx 1.15891 + 0.71869i$}
   \end{subfigure}
    \caption{Approximating parameters for $\lambda \approx 1.16302+ 0.71056 i$. Compare with Figure~\ref{fig:lambda}}
    \label{fig:example}
\end{figure}

This shows that for $d\geq N$, $\sigma^d_{\underline{s}}(B_\epsilon) \subset B_\epsilon$. Since $\sigma^{d}_{\underline{s}}$ is strongly contracting on the compact set $B_\epsilon$, by the Banach fixed point theorem, it has a fixed point in $B_\epsilon$. However, $\sigma^d_\epsilon$ has a unique global fixed point, this thus $[\phi_d] \in B_\epsilon$. Since $\epsilon$ was chosen arbitrarily, we have $[\phi_d] \longrightarrow [\phi]$.
\end{proof}
\begin{proposition}\label{prop:polys exist}
    There exist polynomials $p_{d,\lambda_d} \longrightarrow  p_\lambda$ locally uniformly.
\end{proposition}
\begin{proof}
We may assume that $\varphi_d(\infty) = \varphi(\infty)  = \infty$, $\varphi_d(e_1) = \varphi(e_1)= 0$. Additionally, let $\psi_d  \in [\phi_d]$ be the unique map that satisfies $\phi_d(e_2) = \psi_d(e_2) = \lambda_d$ and  $\phi_d \circ [\mathcal{F}_{\underline{s}}] = p_{d,\lambda_d} \circ \psi_d$. The $p_{d,\lambda_d}$ are pcf for each $d$. Similarly, let $\psi \in [\phi]$ be the unique map that satisfies $\psi(e_2) = \phi(e_2) = \lambda$, and $\phi \circ [\mathcal{F}_{\underline{s}}] = p_\lambda \circ \psi$. To prove this proposition, it suffices to show that $\lambda_d \longrightarrow \lambda$. 

Let $\widetilde{\phi}_d  = \frac{\lambda}{\lambda_d}\phi_d $. We note that $p_{d,\lambda} \circ \psi_d = \widetilde{\phi}_d \circ [\mathcal{F}_{\underline{s}}]$. 

There exist quasiconformal maps $\phi'_d \sim \widetilde{\phi}_d$ rel $A_{\underline{s}}$ and 
$\phi' \sim \phi$ rel $A_{\underline{s}}$ such that the complex dilitation of $\phi'_d \circ (\phi')^{-1} \rightarrow 1$. By Proposition~\ref{prop: qc homeos fixing three points}, $\phi'_d \longrightarrow \phi'$ uniformly on compact sets of $\Gamma_{\underline{s}} \setminus \{\infty\}$. Let $\psi'_d \in [\phi_d]$ and $\psi' \in [\phi]$ be maps such that $p_{d,\lambda} \circ \psi'_d = \phi'_d \circ [\mathcal{F}_{\underline{s}}]$ and $p_{\lambda} \circ \psi' = \phi' \circ [\mathcal{F}_{\underline{s}}]$.  By Proposition~\ref{pullback_convergence}, $\psi'_d \longrightarrow \psi'$ uniformly on compact sets of $\Gamma_{\underline{s}} \setminus \{\infty\}$. 

Given $x \in A_{\underline{s}}$, letting $y=[\mathcal{F}_{\underline{s}}](x)$, we have $\phi'_d(y)  = \widetilde{\phi}_d(y)$,  and $\phi'(y)= \phi(y)$. Therefore $\psi'_d(x) = \psi_d(x)$ and $\psi'(x) = \psi(x)$. In particular,   $\psi'_d(e_2)= \lambda_d$ and $\psi'(e_2) = \lambda$; proving that $\lambda_d \longrightarrow \lambda$. 

\end{proof}

\begin{proof}[Proof of Theorem~\ref{thm:A}]
By Proposition~\ref{prop:polys exist}, there exists a sequence of polynomials $p_{d,\lambda_d}$ that converge to $p_\lambda$ locally uniformly. Let $\underline{s}$ be an external address for $\lambda$, with preperiod $\ell$ and period $k$. The map $p_\lambda$ is Thurston equivalent to $[\mathcal{F}_{\underline{s}}]$ and $p_{d,\lambda_d}$ is Thurston equivalent to $[\mathcal{F}_{d,\theta_d}]$ for each $d$, where $\theta_d$ are as constructed in  Proposition~\ref{prop:address to poly}. By Proposition~\ref{same_spider_exp}, $p_{d,\lambda_d}$ and $p_\lambda$ are conjugate on their postsingular sets for each $d$. Let $c_d \in \mathcal{M}_d$ be the landing point of $R_d(\theta_d)$. Then $z^d+c_d$ is conjugate to $p_{d,\lambda_d}$, and by   Proposition~\ref{prop:address to poly}, we have a polynomial $q \in \Z[x]$ with $\deg q \leq \ell+k-2$ such that $(d-1)\theta_d = \frac{(d-1)q(d)}{d^\ell(d^k-1)}$ for each $d$. 

\end{proof}
\begin{example}
    Let $\lambda = 2\pi i r$, where $r\in \Z \setminus \{0\}$. The orbit of $0$ under $p_\lambda$ is $0\longrightarrow 2\pi i r \righttoleftarrow$, and $\lambda$ has a unique external address, $0\overline{r}$. We have
    \begin{align*}
        \theta_d& = \begin{cases}
            \frac{r}{d(d-1)} & r>0\\
            \frac{d-|r|}{d(d-1)} & r<0\\
         \end{cases}\\
           q(x) &= r\\
          c_d& = e^{\frac{2\pi i r}{d}}\\
           \lambda_d &= d(c_d-1)
    \end{align*}
 \end{example}

\begin{example}
   Let $\lambda $ be the landing point of the ray at address $\underline{s} = 000(-1)\overline{0010}$ (its approximate value is $ 1.16302+ 0.71056 i$; see Figure~\ref{fig:lambda}). The orbit of $0$ under $p_\lambda$ has the form $0 \longrightarrow \lambda \longrightarrow p_\lambda(\lambda) \longrightarrow p_{\lambda}^{\circ 2}(\lambda) \longrightarrow p_{\lambda}^{\circ 3}(\lambda) \longrightarrow p_{\lambda}^{\circ 4}(\lambda) \longrightarrow p_{\lambda}^{\circ 5}(\lambda) \righttoleftarrow$. 

   In Example~\ref{ex:last example}, we computed $\theta_d = \frac{d^5 - d^4 +1}{d^4(d^4-1)}$. In Fig~\ref{fig:lambda} we have indicated the position of $\lambda_d$ for $d = 10, 50, 100$ and $200$, along with approximate values. 
\end{example}

\bibliographystyle{amsalpha}
\bibliography{bibtemplate}
\end{document}